\documentclass[12pt]{amsart}
\usepackage{tikz,array, verbatim}
\usepackage{amsfonts, amsmath, latexsym, epsfig, caption}
\usepackage{amssymb, color}

\usepackage{epsf}
\usepackage{array}
\usepackage{ragged2e}
\usepackage{hyperref}
\usepackage[margin=1in]{geometry}
\usepackage{longtable}

\DeclareMathOperator{\conv}{conv}

\title[On the regularity radius of Delone sets in $\R^3$]{On the regularity radius of Delone sets in $\R^3$}

\begin{document}

\author[N. Dolbilin]{Nikolay Dolbilin}
\address{Nikolay Dolbilin, Steklov Mathematical Institute, 8 Gubkina str., Moscow, 119991, Russia}
\email{dolbilin@mi-ras.ru}

\author[A. Garber]{Alexey Garber}
\address{Alexey Garber, School of Mathematical \& Statistical Sciences, The University of Texas Rio Grande Valley, 1 West University Blvd, Brownsville, TX, 78520, USA}
\email{alexeygarber@gmail.com}

\author[U. Leopold]{Undine Leopold}
\address{Undine Leopold, Department of Mathematics, Northeastern University, Boston, MA, 02115, USA}
\email{undine.leopold@googlemail.com}

\author[E. Schulte]{Egon Schulte}
%\thanks{Egon Schulte was supported by the Simons Foundation Award No. 420718.}
\address{Egon Schulte, Department of Mathematics, Northeastern University, Boston, MA, 02115, USA}
\email{e.schulte@northeastern.edu}

\newcommand{\R}{\ensuremath{\mathbb{R}}}

\newtheorem{theorem}{Theorem}[section]
\newtheorem{proposition}[theorem]{Proposition}
\newtheorem{corollary}[theorem]{Corollary}
\newtheorem{lemma}[theorem]{Lemma}
\newtheorem{problem}[theorem]{Problem}
\newtheorem*{conjecture}{Conjecture}
\newtheorem{question}{Question}
\newtheorem{claim}{Claim}

\theoremstyle{definition}
\newtheorem{definition}[theorem]{Definition}

\theoremstyle{remark}
\newtheorem*{remark}{Remark}

\begin{abstract}
We complete the proof of the upper bound $\hat\rho_3\leq 10R$ for the regularity
radius of Delone sets in three-dimensional Euclidean space. Namely, summing up the results obtained earlier, and adding the missing cases, we show that
if all $10R$-clusters of a Delone set $X$ with parameters $(r,R)$
are equivalent, then $X$ is a regular system.
\end{abstract}

\maketitle

\section{Introduction}
\label{intro}

The standard mathematical model of an ideal crystal involves two fundamental geometric concepts:\ a uniformly distributed discrete point set (Delone set) in Euclidean $d$-space $\mathbb{R}^d$, and a crystallographic group of Euclidean isometries (space group) in $\mathbb{R}^d$ acting on this point set. An ideal crystal structure is modelled by a Delone set that can be expressed as the union of finitely many point orbits under some crystallographic group.

Understanding the appearance of a crystallographic symmetry group in an atomic structure created in a crystallization process has always been one of the fundamental problems in  crystallography. Physicists tend to explain the emergence of the space group by the fact that, in crystalline matter, local arrangements of atoms of the same kind tend to be identical (congruent). The global regularity of the crystal structure then somehow is the result of the local interaction of the atoms, as captured by the following quote from Chapter 30 of Feynman Lectures in Physics \cite{fey}:\ ``When the atoms of matter are not moving around very much, they get stuck together and arrange themselves in a configuration with as low an energy as possible. If the atoms in a certain place have found a pattern which seems to be of low energy, then the atoms somewhere else will probably make the same arrangement. For these reasons, we have in a solid material a repetitive pattern of atoms.'' Thus the emergence of space-group symmetry and periodicity in a crystal structure can be seen as resulting from the congruence of the local atomic arrangements throughout the entire crystal. 

Historically, the mathematical theory describing the relationship between the local and global structure of a crystal was initiated by Delone, Dolbilin, Shtogrin, and Galiulin in \cite{local}. The main goal of this theory is to understand in mathematical terms how congruence of local atomic arrangements forces global structural regularity or periodicity. The mathematical analysis is carried out in terms of Delone sets $X$ with parameters $(r,R)$, where (usually) $2r$ is the smallest interpoint distance of $X$, and $R$ is the radius of a largest ``empty ball'' that can be inserted into the interstices of $X$ in $\mathbb{R}^d$. There are two types of Delone sets that account for regularity or periodicity, respectively, namely a regular system (orbit of a single point under a crystallographic group) or a multi-regular system (union of finitely many point orbits under a crystallographic group). An ideal crystal is a multi-regular system. The local theory proposed in~\cite{local} aims at finding sufficient local conditions for a Delone set $X$ to be a regular system or multi-regular system. In other words, the local theory searches for local conditions on Delone sets $X$ that guarantee the emergence of a crystallographic group of symmetries producing $X$ as an orbit set consisting of a single point orbit or finitely many point orbits, respectively. 

The property whether or not a Delone set in Euclidean space is a regular system or a multi-regular system is invariant under scaling. Thus, beginning with Section~\ref{ltrr}, we will focus entirely on Delone sets with parameters $\frac12$ and $R$ (that is, with $r=\frac12$); these are the sets with distance between points at least $1$.

The {\em regularity radius\/} $\hat\rho_d$ is the smallest positive number $\rho$ with the property that each Delone set $X$ (with parameters $(r,R)$) in $\mathbb{R}^d$ with mutually equivalent point clusters (point neighborhoods) of radius $\rho$ is a regular point system. Here equivalence of clusters means congruence under a center preserving isometry. (For precise definitions see Section~\ref{ltrr}.) Thus $\hat{\rho}_d$ is defined by two properties:\ first, each Delone set $X$ with mutually equivalent point clusters of radius $\hat{\rho}_d$ is a regular system; and second, for any radius $\rho$ smaller than $\hat{\rho}_d$ there exists a Delone set with mutually equivalent clusters of radius $\rho$ which is not a regular system. A priori it is not at all obvious from the definition that this number $\hat{\rho}_d$ exists (but it does!), and how it would depend on the dimension $d$ and the parameters $(r,R)$ associated with Delone sets. The celebrated ``Local Criterion for Regular Systems'' of Delone, Dolbilin, Shtogrin \& Galiulin~\cite{local} (our Theorem~\ref{thm:local} below) implies the existence of an upper bound for $\hat{\rho_d}$; in fact, even the existence of an upper bound of the form ${\hat\rho}_d\leq c(d,R/r)R$, with some constant $c(d,R/r)$ depending on $d$ and $R/r$ (for details see also \cite{Dol2018a}). A more explicit upper bound was established in Dolbilin, Lagarias \& Senechal~\cite{DLS1998}, namely
\[ \hat{\rho_d} \leq 2R(d^2+1) \log_2(2R/r+2).\]

A main goal is to find good upper and lower estimates for the regularity radius $\hat{\rho}_d$.  For dimensions~$1$ and $2$ the exact values are known: ${\hat \rho}_1=2R$ and ${\hat\rho}_2=4R$. For $d=1$ the verification is straightforward, and for $d=2$ the proof follows from the Local Criterion as well as the general lower bound ${\hat \rho }_d\geq 4R$, $d\geq 2$,  described in \cite{Dol2015, Dol2018a, LowerBound} (or its generalization ${\hat\rho }_d\geq 2dR$, $d\geq 2$, proved in~\cite{LowerBound}).

The purpose of the present paper is to complete the proof of the new upper bound ${\hat\rho}_3\leq 10R$ in dimension~3. The long and somewhat involved history of the estimate $4R\leq {\hat\rho}_3\leq 10R$ is described in more detail in Section~\ref{hist}. The search for an upper bound started off with a simple but crucial observation by M.~Shtogrin in the 1970's, which was only published quite recently~\cite{Sto2010}. This observation limits the possible choices of finite groups of isometries (finite subgroups of the orthogonal group $\mathrm{O}(3)$) that can occur as symmetry groups of clusters of radius $2R$ ($2R$-cluster groups) in a Delone set in $\mathbb{R}^3$ with mutually equivalent $2R$-clusters. In fact, in any such group, the order of any rotation axis cannot exceed 6. Shtogrin's simple necessary condition reduces the possibilities for $2R$-cluster groups from an infinite list of finite subgroups of $\rm{O}(3)$ to a finite list which we denote by $L_6$. Quite recently, a more general result on $2R$-cluster groups in arbitrary Delone sets in $\mathbb{R}^3$ was found in~\cite{Dol2019} (Theorem~\ref{thm:rotation2} below), which immediately implies Shtogrin's result. 

In light of the Local Criterion of Theorem~\ref{thm:local}, and the fact that the tower height of each group $G$ (the number of subgroups of the longest chain of strictly nested subgroups of $G$) from $L_6$ is bounded by 6, the tower bound criterion of Theorem~\ref{thm:tower} immediately establishes the upper estimate ${\hat\rho}_3\leq 14R$. For the groups $G$ from $L_6$ with tower height at most 4, the regularity of a Delone set $X$ of the desired kind follows from the Local Criterion. In order to lower this upper bound from $14R$ to $10R$ by means of the tower bound criterion, one needs to analyze the groups $G$ from $L_6$ with tower height larger than 4. For each such group $G$, it must be proved that either $G$ cannot occur as a $2R$-cluster group in a Delone set in $\mathbb{R}^3$ with mutually equivalent $2R$-clusters, or that otherwise the mutual equivalence of $10R$-clusters in any such Delone set with $2R$-cluster group $G$ forces the regularity of $X$. 

In Section~\ref{ltrr} we review basics about Delone sets, as well as the Local Criterion and the Tower Bound. Section~\ref{2Rregcond} revisits known $2R$-regularity conditions and gives a proof of a generalization of Shtogrin's condition. In Section~\ref{S8D4d}, two particular groups are rejected as possible $2R$-cluster groups. This then enables us in Section~\ref{10rbound} and Appendix~\ref{appendix} to complete the proof. 

\subsection*{Acknowledgment}
We would like to thank the American Institute for Mathematics (AIM) for hosting a weeklong workshop in 2016 on ``Soft Packings, Nested Clusters, and Condensed Matter'', as well as the authors' ongoing SQuaRe research project on ``Delaunay Sets: Local Rules in Crystalline Structures'' that grew out of it. The present paper resulted from the discussions at the first SQuaRE meeting at AIM in November 2018. We greatly appreciated the opportunity to meet at AIM and are grateful to AIM for its hospitality. The work of E.~S. was partially supported by Simons Foundation award no.~420718. 

\section{History and status of the proof}
\label{hist}

The list $L_6$ of theoretically possible $2R$-cluster groups consists precisely of the finite subgroups of $\rm{O}(3)$ which do not contain rotations of order exceeding $6$. A detailed investigation of the groups $G$ from $L_6$ with tower height larger than 4 was carried out independently by M.~Shtogrin and N.~Dolbilin over many years, and led each, independently, to the conclusion that ${\hat\rho}_3\leq 10R$ should hold. However a complete proof of this upper estimate remained unpublished. The most significant attempt at a complete proof was made in \cite{Dol2018}. As our paper critically uses several results of \cite{Dol2018}, we briefly review the approach described there. 

In~\cite{Dol2018}, the list $L_6$ was split into the following four sublists, to some degree  overlapping.

{\it Sublist~1.\ Groups of $L_6$ which contain the central symmetry.}\  By Theorem~\ref{thm:antipodal} below (see  \cite{Dol2015,Dol2018a,DM2015}), for each such group $G$, if $X$ is any Delone set with mutually equivalent $2R$-clusters and with $2R$-cluster group $G$, then $X$ is a regular system.

{\it Sublist~2.\ Groups of $L_6$ with a rotation axis of order 6.}\  It was announced in~\cite{Sto2010} that, for each such group $G$, if $X$ is any Delone set with mutually equivalent $2R$-clusters and with $2R$-cluster group $G$, then $X$ is a regular system. A full proof is given in this paper (see Theorem~\ref{thm:n=6}).

{\it Sublist~3. Groups of $L_6$ without the central symmetry and with tower height $5$ or $6$.}\  Sublist 3 has been of particular interest, as the Local Criterion applied to these groups gives only ${\hat \rho}_3\leq 14R$ and the methods for groups containing the central symmetry do not work in this case. Sublist~3 contains, for instance, the rotational symmetry groups of the regular tetrahedron, cube, and icosahedron (of orders 12, 24, or 60), denoted $T$, $O$, or $I$, respectively, as well as the full symmetry group of the regular tetrahedron (of order 24), denoted $T$. In \cite{Dol2018}, various geometric arguments were used to deal with all groups on Sublist 3, save one (see below). Again it was shown for all these groups $G$ (save the exceptional one) that, if $X$ is a Delone set with mutually equivalent $2R$-clusters and with $2R$-cluster group $G$, then $X$ is a regular system. The discussion of the groups on Sublist~3  is the most important result of \cite{Dol2018}. From \cite{LowerBound} we know that in order to conclude regularity of a Delone set for some $2R$-cluster groups $G$, one needs to require the equivalence of at least the $6R$-clusters, even if the group has a relatively small tower height. Results of \cite{Dol2018} are of particular interest to us, as for some groups $G$ on list $L_6$, regardless of their large tower heights 5 or 6, the equivalence of $2R$-clusters already suffices to conclude regularity.     

{\it Sublist~4. Groups of $L_6$ not on Sublists 1, 2 or 3.}\ These groups do not contain the central symmetry and have no rotation axis of order 6. As their tower height is at most 4, they do not require particular attention in the proof of the $10R$-bound and hence were not considered in \cite{Dol2018}.

The exceptional group of Sublist 3 not covered by~\cite{Dol2018} is the symmetry group of the ``regular'' antiprism over the square, $D_{4d}$ (in Sch\"onflies notation -- see Section~\ref{10rbound}), which has order 16 and tower height 5. By accident this group was omitted from the analysis in~\cite{Dol2018}, as was recently noticed by A.~Garber. Thus the hoped-for completion of the proof of ${\hat\rho}_3\leq 10R$ unfortunately was not quite accomplished in~\cite{Dol2018}. 

The main goal and result of this paper is to deal with the case $D_{4d}$ and complete the proof of the estimate ${\hat\rho}_3\leq 10R$. In particular, we prove that there is no Delone set $X$ in $\mathbb{R}^3$ with mutually equivalent $2R$-clusters and with $2R$-cluster group $D_{4d}$. This closes the gap in \cite{Dol2018} and completes the proof of the $10R$ estimate. We also present a proof of the regularity of a Delone set with mutually equivalent $2R$-clusters and with a $2R$-cluster group from Sublist~2, as no full proof of this fact is available in the literature. Our approach will also settle the case of the cyclic subgroup $S_8$ of $D_{4d}$ of order 8 generated by a rotatory reflection, which has tower height $4$ and is on Sublist~4. As for $D_{4d}$, there is no Delone set $X$ in $\mathbb{R}^3$ with mutually equivalent $2R$-clusters and with $S_8$ as $2R$-cluster group.

\section{Local theory and the regularity radius}
\label{ltrr}

Points, or vectors, in $\mathbb{R}^d$ are denoted by bold-faced letters. For $\mathbf{x},\mathbf{y}\in\mathbb{R}^d$ we set $\mathbf{xy}:=\mathbf{y}-\mathbf{x}$ and write $|\mathbf{xy}|$ for the length of $\mathbf{xy}$. The line passing through two distinct points $\mathbf{x},\mathbf{y}\in\mathbb{R}^d$ is denoted by $\overline{\mathbf{xy}}$, and the line segment between $\mathbf{x}$ and $\mathbf{y}$ by $[\mathbf{x},\mathbf{y}]$. If ${\mathbf x}\in\mathbb R^d$ and $\rho\geq 0$, we let $B_{\mathbf x}(\rho)$ denote the closed $d$-ball in $\mathbb R^d$ of radius $\rho$ centered at ${\mathbf x}$.

In this section, we introduce the main concepts of the local theory of regular systems and survey the main results about the regularity radius. Although we are primarily interested in the case $d=3$, our discussion in this section is for arbitrary dimension $d$. 

\begin{definition}
\label{def1}
Let $r$ and $R$ be positive real numbers with $r\leq R$. A subset $X$ of $\mathbb{R}^d$ is called a \emph{Delone set with parameters $(r,R)$} if the following two conditions hold:
\\
(1) each open $d$-ball of radius $r$ contains at most one point of $X$;
\newline
(2) each closed $d$-ball of radius $R$ contains at least one point of $X$.
\end{definition}

In addition, we adopt the convention that in designating the parameters $(r,R)$ to a Delone set $X$ we choose the largest possible value of $r$ and the smallest possible value of $R$ that satisfy the properties of Definition~\ref{def1}. 

\begin{definition}\label{def:deloneset} A Delone set $X$ in $\mathbb{R}^d$ is called a {\em regular system\/} if its symmetry group $S(X)$ acts transitively on $X$, i.e., for any pair of points ${\mathbf x}, {\mathbf x}'$ in $X$ there exists an isometry $g\in S(X)$ such that $g({\mathbf x})={\mathbf x}'$. 
\end{definition}

Thus a regular system $X$ in $\mathbb{R}^d$ coincides with the orbit of any one of its points under its symmetry group $S(X)$. 

The property of whether or not a Delone set is a regular system is invariant under similarity transformations of $\mathbb{R}^d$. More explicitly, if $X$ is a Delone set with parameters $(r,R)$ which is a regular system, and if $g$ is a transformation of $\mathbb{R}^d$ of the form $g({\mathbf x})=\lambda\sigma({\mathbf x})+{\mathbf t}$ with $\lambda\neq 0$, $\sigma\in \mathrm O(d)$, and ${\mathbf t}\in\mathbb{R}^d$, then $g(X)$ is a Delone set with parameters $(\lambda r,\lambda R)$ which is also a regular system. Thus for the purpose of finding sufficient conditions that guarantee regularity, it suffices to consider only Delone sets with~$r=\frac12$, and we will do this from now on. 

Accordingly we have the following revised definition of a Delone set.

\begin{definition}Let $R\geq\frac12$. A subset $X$ of $\mathbb{R}^d$ is called a Delone set if the following two conditions hold:\\
(1) each open $d$-ball of radius $\frac12$ contains at most one point of $X$;
\newline
(2) each closed $d$-ball of radius $R$ contains at least one point of $X$.
\end{definition}

We shall primarily work with the second definition. In particular, unless said otherwise, we will always assume that $r=\frac12$. We also adopt a similar convention as above for the choice of the parameter $R$ (and the parameter $\frac12$). 

Given a Delone set $X$ in $\mathbb{R}^d$, a closed ball $B$ in $\mathbb{R}^d$ is called an {\em empty ball\/} of $X$ if no point of $X$ lies in the interior of $B$. Note that an empty ball of $X$ may have points of $X$ on its boundary. By our conventions, if $X$ is a Delone, then the parameter $R$ is the radius of the largest empty ball of $X$, and distances between pairs of  points of $X$ are at least $1$ and can get arbitrarily close to $1$.

In the local theory of regular systems, the key concepts are that of a cluster and its symmetry group. For a point ${\mathbf x}$ of a Delone set $X$ and for $\rho\geq 0$, we call 
\[C_{\mathbf x}(\rho):=X\cap B_{\mathbf x}(\rho)\] 
the {\em cluster of radius $\rho$} of $X$ with \emph{center} ${\mathbf x}$, or simply the {\em $\rho$-cluster of $X$ at ${\mathbf x}$}, and 
\[H_{\mathbf x}(\rho):=X\cap \partial B_{\mathbf x}(\rho)\] 
the $\rho$-{\it shell} of $X$ with \emph{center} ${\mathbf x}$. Clusters and shells of Delone sets are finite point sets. The $\rho$-clusters $C_{\mathbf x}(\rho)$ and $C_{{\mathbf x}'}(\rho)$ at two points ${\mathbf x},{\mathbf x}'$ of $X$ are said to be \emph{equivalent} if there exists an isometry $g$ of $\mathbb{R}^d$ such that $g({\mathbf x})={\mathbf x}'$ and $g(C_{\mathbf x}(\rho))=C_{{\mathbf x}'}(\rho)$. Note that equivalence of clusters is stronger than mere congruence of clusters, as the isometry $g$ furnishing the equivalence must map the center $\mathbf x$ of $C_{\mathbf x}(\rho)$ to the center ${\mathbf x}'$ of $C_{{\mathbf x}'}(\rho)$. 

Given a Delone set $X$ and $\rho\geq 0$, the set of all $\rho$-clusters of $X$ is partitioned into classes of equivalent $\rho$-clusters. If the set of equivalence classes of $\rho$-clusters is finite for every $\rho>0$, then $X$ is said to be of \emph{finite type}. It is known that if the number of equivalence classes of $2R$-clusters in a Delone set $X$ is finite, then $X$ is of finite type. For a Delone set $X$ of finite type, we let $N(\rho)$ denote the number of equivalence classes of $\rho$-clusters, and call $N$ the {\em cluster counting function\/} of $X$. Then $N$ is a non-decreasing function. Clearly, $N(\rho)=1$ whenever $\rho<1$, since then $C_{\mathbf x}(\rho)=\{\mathbf x\}$ for each ${\mathbf x}\in X$. It follows from the Local Regularity Criterion, Theorem \ref{thm:local} below, the regularity property of a Delone set $X$ can also be described in terms of its cluster counting function:\ $X$~is a regular system if and only if $N(\rho)= 1$ for each $\rho\geq 0$. 

The first major question in the local theory asks if the regularity of a Delone set can be recognized on clusters of bounded radius:\ does there exist a positive number $\rho$ with the property that each Delone set $X$ with mutually equivalent $\rho$-clusters (that is, $N(\rho) = 1$) must necessarily be a regular system. A positive answer is provided by the Local Criterion of Theorem~\ref{thm:local} and its consequences. The smallest such number $\rho$, denoted $\hat{\rho_d}=\hat{\rho_d}(R)$, is called the {\em regularity radius\/} and a priori depends on the dimension $d$ and the parameter $R$.

The Local Criterion requires the notion of a cluster group. The \emph{cluster group\/} $S_{\mathbf x}(\rho)$ of a $\rho$-cluster $C_{\mathbf x}(\rho)$, or simply the \emph{$\rho$-cluster group\/} $S_{\mathbf x}(\rho)$ at $\mathbf x$, in a Delone set $X$ is defined as the stabilizer of $\mathbf x$ in the full symmetry group of $C_{\mathbf x}(\rho)$. Thus $S_{\mathbf x}(\rho)$ consists of all isometries $g$ of $\mathbb{R}^d$ such that $g({\mathbf x})={\mathbf x}$ and $g(C_{\mathbf x}(\rho))=C_{\mathbf x}(\rho)$. Clearly, if $\rho< 1$, then $C_{\mathbf x}(\rho)=\{\mathbf x\}$ and thus $S_{\mathbf x}(\rho)$ is isomorphic to the full orthogonal group $\rm{O}(d)$ of~$\mathbb{R}^d$. More generally, if the affine hull of a cluster has dimension at most $d-2$, then its cluster group (consisting of $d$-dimensional isometries) is necessarily infinite. However, if the affine hull of a cluster has dimension at least $d-1$, then the cluster group is a finite group, since clusters are finite. As shown in \cite{local}, in a Delone set $X$, the $\rho$-clusters for $\rho\geq 2R$ are $d$-dimensional and thus have finite cluster groups. The cluster groups for a given $X$ are non-increasing as the radius $\rho$ grows, that is, $S_{\mathbf x}(\rho)\supseteq S_{\mathbf x}(\rho')$ whenever $\rho\leq\rho'$, and so in particular, $S_{\mathbf x}(\rho)$ is a subgroup of the finite group $S_{\mathbf x}(2R)$ if $\rho\geq 2R$. 

Observe also that the cluster groups of any two equivalent $\rho$-clusters are conjugate subgroups of $\rm{Iso}(d)$, the group of all isometries of $\mathbb{R}^d$. In the presence of mutual $\rho$-cluster equivalence we often de-emphasize the center $\mathbf x$ in the notation for a cluster and simply write $S_X(\rho)$ instead of $S_{\mathbf x}(\rho)$ for the cluster group, with the understanding that $S_X(\rho)$ is only defined up to conjugacy in $\rm{Iso}(d)$.

\begin{theorem}
[Local Regularity Criterion, \cite{local}]\label{thm:local} 
A Delone set $X$ in $\mathbb{R}^d$ is a regular system if and only if $X$ satisfies the following two conditions for some $\rho_0>0$ and some ${\mathbf x}_0\in X$:
$$N(\rho_0 + 2R) = 1, \eqno(2.1)$$
$$S_{{\mathbf x}_{0}}(\rho_0) = S_{{\mathbf x}_{0}}(\rho_0 + 2R). \eqno(2.2)$$
Moreover, in this situation, $S_{\mathbf x}(\rho_0) = S_{\mathbf x}(\rho_0 + 2R)$ for all ${\mathbf x}\in X$; the cluster groups stabilize at radius $\rho_0$, that is, $S_{\mathbf x}(\rho) = S_{\mathbf x}(\rho_0)$ for all ${\mathbf x}\in X$ and all $\rho\geq\rho_0$; and the regular system $X$ is uniquely determined by the cluster $C_{{\mathbf x}_{0}}(\rho_0 + 2R)$.
\end{theorem}

Theorem~\ref{thm:local} immediately implies the following Theorem~\ref{thm:asym} which deals with a special case.  Historically, however, Theorem~\ref{thm:asym} has preceded Theorem~\ref{thm:local} and has inspired it. Call a Delone set $X$ {\em locally asymmetric\/} if the cluster group $S_{\mathbf x}(2R)$ is trivial for each ${\mathbf x}\in X$. 

\begin{theorem}\label{thm:asym} 
If $X$ is a locally asymmetric Delone set in $\mathbb{R}^d$ with mutually equivalent $4R$-clusters, then $X$ is a regular system.
\end{theorem}

For a finite group $G$, a finite chain of distinct nested subgroups,
$$G:= G_1\supset G_2\supset\ldots\supset G_{m-1}\supset G_m = \{1\},$$
is called a {\em tower\/} of $G$. Note that towers of $G$ begin with $G$ and end with the trivial group. The number $m$ of subgroups in the tower is called the (\emph{tower}) \emph{height} of the tower. By $m(G)$ we denote the maximal height of a tower of $G$, and note that $G$ may have several towers of height $m(G)$. It is clear that $m(G)\leq \Omega +1\leq \log_2 2|G|$, where $\Omega=\Omega(|G|)$ denotes the total number of prime factors of $|G|$ (counted with multiplicities). 

The next theorem is a reformulation of \cite[Prop. 2.1]{Dol2018}. It expresses the regularity of a Delone set in terms of the prime factorization of the order of its $2R$-cluster group, and thus provides an important tool for finding bounds for the regularity radius.

\begin{theorem}
[Tower Bound]\label{thm:tower}
Let $X$ be a Delone set in $\mathbb{R}^d$ with mutually equivalent $2R$-clusters, and suppose $\Omega$ is the number of prime factors of the order $|S_{X}(2R)|$ (counted with multiplicities). If all $2(\Omega +2)R$-clusters of $X$ are mutually equivalent, then $X$ is a regular system.
\end{theorem}
\begin{remark}
In terms of the highest tower for $G=S_X(2R)$, the theorem can be reformulated as follows. If all $2(m(G)+1)R$-clusters of $X$ are mutually equivalent, then $X$ is a regular set.
\end{remark}

\begin{proof}
Suppose all $2(\Omega+2)R$-clusters of $X$ are mutually equivalent, and let ${\mathbf x}\in X$. Then $S_X(\rho)=S_{\mathbf x}(\rho)$ for all $\rho\leq 2(\Omega +2)R$. Consider the chain of embedded clusters, 
$$C_{\mathbf x}(2R)\subset C_{\mathbf x}(4R)\subset C_{\mathbf x}(6R)\subset \ldots \subset C_{\mathbf x}(2(\Omega+1)R)\subset C_{\mathbf x}(2(\Omega+2)R),$$
as well as the corresponding chain of cluster groups,
$$S_{\mathbf x}(2R)\supseteq S_{\mathbf x}(4R)\supseteq S_{\mathbf x}(6R)\supseteq \ldots \supseteq S_{\mathbf x}(2(\Omega+1)R)\supseteq S_{\mathbf x}(2(\Omega+2)R),$$
where here some pairs of consecutive groups may coincide. Now, if all cluster groups in this chain are distinct, then the index of each cluster group in the preceding cluster group is strictly larger than~1 and therefore $|S_{\mathbf x}(2R)|$ is a product of at least $\Omega+1$ prime numbers, contrary to our assumption. Thus at least one pair of consecutive cluster groups in the chain must coincide, that is, $S_{\mathbf x}(2j_0R)= S_{\mathbf x}(2(j_0+1)R)$ for some $j_0$. But then Theorem~\ref{thm:local}, applied with $\rho_{0}=2j_0R$, shows that $X$ is regular. 
\end{proof} 

Further, we require the following theorem on locally antipodal sets which is interesting and important in its own right. We call $X$ \textit{locally antipodal} if, for each ${\mathbf x}\in X$, 
the $2R$-cluster $C_{\mathbf x}(2R)$ is centrally symmetric with respect to ${\mathbf x}$, i.e. the cluster group $S_{\mathbf x}(2R)$ contains the central symmetry with respect to ${\mathbf x}$.

\begin{theorem}[\cite{Dol2015}, see also \cite{DM2015}]\label{thm:antipodal}
If $X$ is a locally antipodal Delone set in $\mathbb{R}^d$, then $X$ is a regular system.
\end{theorem}

\section{$2R$-regularity conditions for dimension $3$}
\label{2Rregcond}

In this section, we survey some known results about local conditions for the regularity of a Delone set in $\mathbb{R}^3$. Throughout, we assume that $X$ is a Delone set in $\mathbb{R}^3$ with mutually equivalent $2R$-clusters, i.e. $N(2R)=1$. Then any two $2R$-cluster groups are conjugate in $\rm{Iso}(3)$, and each $2R$-cluster group is conjugate in $\rm{Iso}(3)$ to a finite subgroup of $\rm{O}(3)$. The finite subgroups of $\rm{O}(3)$ are well-known and are described in Section~\ref{10rbound} using Sch\"onflies notation.  
\medskip

The following Theorem~\ref{thm:rotation} selects, from the infinite list of finite subgroups of $\rm{O}(3)$, a finite sublist, denoted $L_6$, of possible groups which can occur as $2R$-cluster groups in a Delone set with mutually equivalent $2R$-clusters. Then, in conjunction with Theorem~\ref{thm:tower}, inspection of the groups on $L_6$ immediately gives the upper estimate ${\hat \rho}_3\leq 14R$. Our goal is to reduce the upper bound to $10R$. 

\begin{theorem}[Shtogrin]
\label{thm:rotation} 
Suppose that $X$ is a Delone set in $\mathbb{R}^3$ with $N(2R)=1$. If $n$ is the order of a rotation in $S_{X}(2R)$, then $n\leq 6$. 
\end{theorem}

Although the important Theorem~\ref{thm:rotation} was already discovered by Shtogrin in the late 1970's, it was only published in 2010 in an Abstract for a conference in honor of Delone (see \cite{Sto2010}). Dolbilin~\cite{Dol2019} recently observed that Theorem~\ref{thm:rotation} follows from a more general result, Theorem~\ref{thm:rotation2} below, which can be proved for arbitrary Delone sets, and in particular without any assumption on $N(2R)$ whatsoever (including finiteness). Here we present a proof of Theorem~\ref{thm:rotation2} and then derive Theorem~\ref{thm:rotation} as a consequence. 

For a point $\mathbf x$ of a Delone set $X$ in $\mathbb{R}^3$, we let $n({\mathbf x})$ denote the maximal order of a rotation in the $2R$-cluster group $S_{\mathbf x}(2R)$, that is, the maximal order of a rotation axis through $\mathbf x$ for the cluster $C_{\mathbf x}(2R)$.  Clearly, $S_{\mathbf x}(2R)$ may have many rotation axes of order $n({\mathbf x})$.

\begin{theorem}\label{thm:rotation2}
For each Delone set $X$ in $\mathbb{R}^3$ there exists a point ${\mathbf x}\in X$ such that $n({\mathbf x})\leq 6$.	
\end{theorem}

\begin{proof} 
First recall that the $2R$-clusters $C_{\mathbf x}(2R)$ of points $\mathbf x$ in $X$ are full-dimensional and in particular contain points off the rotation axis of any rotation belonging to the $2R$-cluster group $S_{\mathbf x}(2R)$. 

Assume to the contrary that $n({\mathbf x})\geq 7$ for all $\mathbf x\in X$, and choose an arbitrary point $\mathbf x_0^*$ of~$X$. Our goal is to construct an infinite sequence of points  of $X$ such that the corresponding sequence of successive interpoint distances monotonically tends to $0$. More specifically, in this sequence we have  
$|{\mathbf x}_{i+1}^*{\mathbf x}_{i+2}^*|<0.87 |{\mathbf x}_i^*{\mathbf x}_{i+1}^*|$ 
for each $i\geq 0$. Clearly this phenomenon cannot occur in a Delone set, since interpoint distances in $X$ cannot be smaller than $2r(=1)$. 

We begin with the chosen point ${\mathbf x}_{0}^{*}$ and construct ${\mathbf x}_{1}^{*}$. Let $l_{{\mathbf x}_{0}^{*}}$ denote the rotation axis of any rotation in $S_{{\mathbf x}_{0}^{*}}(2R)$  of order $n({\mathbf x}_{0}^{*})$, and let ${\mathbf x}_{1}^{*}$ be a point of $C_{{\mathbf x}_{0}^{*}}(2R)$ lying off $l_{{\mathbf x}_{0}^{*}}$, such that ${\mathbf x}_{1}^{*}$ has the smallest distance from ${\mathbf x}_{0}^{*}$  among all points of $C_{{\mathbf x}_{0}^{*}}(2R)$ lying off~$l_{{\mathbf x}_{0}^{*}}$. Set $r_1^*:=|{\mathbf x}_{0}^*{\mathbf x}_{1}^{*}|$. 

Now observe that the shell $H_{{\mathbf x}_{0}^*}(r_1^*)$ contains the vertices of a convex regular $n_{{\mathbf x}_{0}^*}$-gon~$P_1$ with one vertex at ${\mathbf x}_{1}^{*}$ and with center on $l_{{\mathbf x}_{0}^{*}}$. As the circumradius of $P_1$ does not exceed $r_1^*$, and $n_{{\mathbf x}_{0}^*}\geq 7$, the sidelength $a_1$ of $P_1$ must satisfy the inequality
\begin{equation}
\label{eqfora1}
a_1\leq 2r^*_1\sin \frac{\pi }{n_{{\mathbf x}_{0}^*}}\leq 2r^*_1\sin \frac{\pi}{7}< 0.87~ r^*_1. 
\end{equation}  
Thus $a_1< 0.87\,r^*_1$. Clearly, the two vertices of $P_1$ adjacent to ${\mathbf x}_1^*$ have distance $a_1$ from ${\mathbf x}_1^*$, and together with ${\mathbf x}_1^*$ form a non-collinear triple of points.

Now consider the $2R$-cluster centered at ${\mathbf x }^*_1$, as well as the rotation axis $l_{{\mathbf x}_{1}^{*}}$ of any rotation in $S_{{\mathbf x}_{1}^{*}}(2R)$ of order $n({\mathbf x}_{1}^{*})$. Note that there is no requirement of mutual equivalence of $2R$-clusters among the assumptions of our theorem, so $C_{{\mathbf x}^*_1}(2R)$ and $C_{{\mathbf x}^*_0}(2R)$ may not be congruent, and in particular, $n({\mathbf x}_{1}^{*})$ and $n({\mathbf x}_{0}^{*})$ may be different. By what was said above, at least one of the two neighboring vertices of ${\mathbf x}_1^*$ in $P_1$, denoted ${\mathbf x}_1'$, cannot lie on the axis~$l_{{\mathbf x}_{1}^{*}}$. 

Next choose a point ${\mathbf x}_{2}^{*}$ of $C_{{\mathbf x}_{1}^{*}}(2R)$ not contained in $l_{{\mathbf x}_{1}^{*}}$, such that ${\mathbf x}_{2}^{*}$ has the smallest distance from ${\mathbf x}_{1}^{*}$  among all points of $C_{{\mathbf x}_{1}^{*}}(2R)$ lying off $l_{{\mathbf x}_{1}^{*}}$. Define $r^*_2:=|{{\mathbf x}^*_1}{{\mathbf x}^*_2}|$. Then, since ${\mathbf x}_1'$ lies in $C_{{\mathbf x}_{1}^{*}}(2R)$, it is clear that 
\begin{equation}
\label{eqforr2star}
r_2^*\leq |{\mathbf x}_1^*{\mathbf x}_1'|=a_1< 0.87\,r_1^*.
\end{equation}                 
Further, the point ${\mathbf x}^*_2$ of $C_{{\mathbf x}_1^*}$ determines a regular $n_{{\mathbf x}^*_2}$-gon $P_2$ with one vertex at ${\mathbf x}_{2}^{*}$ and with center on $l_{{\mathbf x}_{1}^{*}}$. Let $a_2$ denote the edge length of $P_2$. Then the two vertices of $P_2$ adjacent to vertex ${\mathbf x}_{2}^{*}$ are at distance $a_2$ from ${\mathbf x}_{2}^{*}$, and together with ${\mathbf x}_{2}^{*}$ form a non-collinear triple.

We now proceed in the same fashion with ${\mathbf x}_{1}^{*}$ and $l_{{\mathbf x}_{1}^{*}}$ replaced by ${\mathbf x}_{2}^{*}$ and $l_{{\mathbf x}_{2}^{*}}$, respectively, and construct a point ${\mathbf x}_{3}^{*}$ nearest to ${\mathbf x}_{2}^{*}$ among all points of $C_{{\mathbf x}_2^*}$ not lying on $l_{{\mathbf x}_{2}^{*}}$, and so on. As $n_{{\mathbf x}^*_2}\geq 7$ this then gives
\begin{equation}
\label{eqforr3star}
r_3^*\leq a_2 < 0.87\,r^*_2 <  (0.87)^2\,r^*_1.
\end{equation}

Hence, continuing in this manner we arrive at the desired infinite sequence of points $\big({\mathbf x}_i^*\big)_{i\geq 0}$, as well as the corresponding sequence of interpoint distances $\big(r_{i+1}^*\big)_{i\geq 0}$ satisfying $r_{i+1}^{*}<0.87 r_{i}^{*}$ for each $i\geq 1$. Thus, for large enough $i$, the interpoint distance $r_{i}^{*}$ is positive but less than $1$. This is impossible in a Delone set. 

Thus there must be a point $\mathbf x$ in $X$ with $n({\mathbf x})\leq 6$. This completes the proof.
\end{proof}

If the $2R$-clusters of a Delone set $X$ are mutually equivalent, that is, $N(2R)=1$, then the maximal rotation orders $n_{\mathbf x}$ at points ${\mathbf x}$ of $X$ must all coincide.  
Thus Theorem~\ref{thm:rotation} follows directly from Theorem~\ref{thm:rotation2}.  

\medskip

%Our next theorem establishes the regularity of Delone sets with mutually equivalent $2R$-clusters, whose $2R$-cluster groups contain 6-fold rotations. This result was announced in~\cite{Sto2010} without proof.

Now we prove a theorem that establishes the regularity of Delone sets with mutually equivalent $2R$-clusters, whose groups contain 6-fold rotations. This result was announced in~\cite{Sto2010} without proof.

\begin{theorem}[Shtogrin, \cite{Sto2010}] 
\label{thm:n=6} 
Let $X$ be a Delone set in ${\mathbb R}^3$ with mutually equivalent $2R$-clusters, and suppose the $2R$-cluster group $S_{X}(2R)$ has a rotation axis of order 6. Then $X$ is a regular system. Moreover, $X$ is (a translate of) either a lattice $\Gamma$ or a bi-lattice $\Gamma\cup(\Gamma+\mathbf t)$, ${\mathbf t}\not\in\Gamma$, where in either case the underlying lattice $\Gamma$ is spanned by a hexagonal lattice in a plane and a vector  orthogonal to the plane (that is, $\Gamma$ has a basis with a Gram matrix of the form 
\[
\begin{bmatrix} 
\lambda   & \lambda /2 &0 \\
\lambda /2 & \lambda & 0 \\
0 & 0 & \mu   	
\end{bmatrix}
\]
with $\lambda,\mu>0$).
\end{theorem}  

\begin{proof}  
For a point ${\mathbf x}$ of $X$, let $l_{\mathbf x}$ denote the 6-fold rotation axis of the $2R$-cluster group $S_{\mathbf x}(2R)$ at $x$. It follows from the complete list of finite subgroups of $\mathrm{O}(3)$ (see Section \ref{10rbound}) that $S_{\mathbf x}(2R)$ can only have one $6$-fold rotation axis. 

Choose an arbitrary point ${\mathbf x}$ of $X$, and keep it fixed. Let ${\mathbf y}_1$ denote the closest point to ${\mathbf x}$ among all points of $X$ off $l_{\mathbf x}$. Set $r_1:=|{\mathbf x}{\mathbf y}_1|$. Since the $2R$-cluster $S_{\mathbf x}(2R)$ is full-dimensional, it is clear that $r_1\leq 2R$. Therefore the shell $H_{\mathbf x}(r_1)$ contains the vertices ${\mathbf y}_1,{\mathbf y}_{2},\ldots,{\mathbf y}_{6}$ of a regular hexagon $P$ with side length $a:=|{\mathbf y}_{1}{\mathbf y}_{2}|$, which is centered on $l_{\mathbf x}$ and contained in a plane perpendicular to $l_{\mathbf x}$. We claim that this plane must pass through ${\mathbf x}$ itself, so that ${\mathbf x}$ becomes the center of $P$. In fact, since the $2R$-clusters at ${\mathbf x}$ and ${\mathbf  y}_1$ are equivalent, the closest point to ${\mathbf y}_1$ among all points of $C_{{\mathbf y}_1}(2R)$ off $l_{{\mathbf y}_1}$ must have distance $r_{1}$ from ${\mathbf y}_1$. Hence, if the line segment $[{\mathbf x},{\mathbf  y}_1]$ is not perpendicular to $l_{\mathbf x}$, then $r_1>a$ and therefore at least one of the points ${\mathbf y}_{2}$ or ${\mathbf y}_{6}$ is a point of $C_{{\mathbf y}_1}(2R)$ lying off $l_{{\mathbf y}_1}$, with a smaller distance from ${\mathbf y}_1$ than $r_{1}$, namely $a$. Thus $P$ lies in a plane through $\mathbf x$ perpendicular to $l_{\mathbf x}$, and we denote this plane by $\Pi:=\Pi_{\mathbf x}$ from now on (see  Fig.  \ref{fig:n=6}).

\begin{figure}[!ht]
\begin{center}
\includegraphics[width=0.7\textwidth]{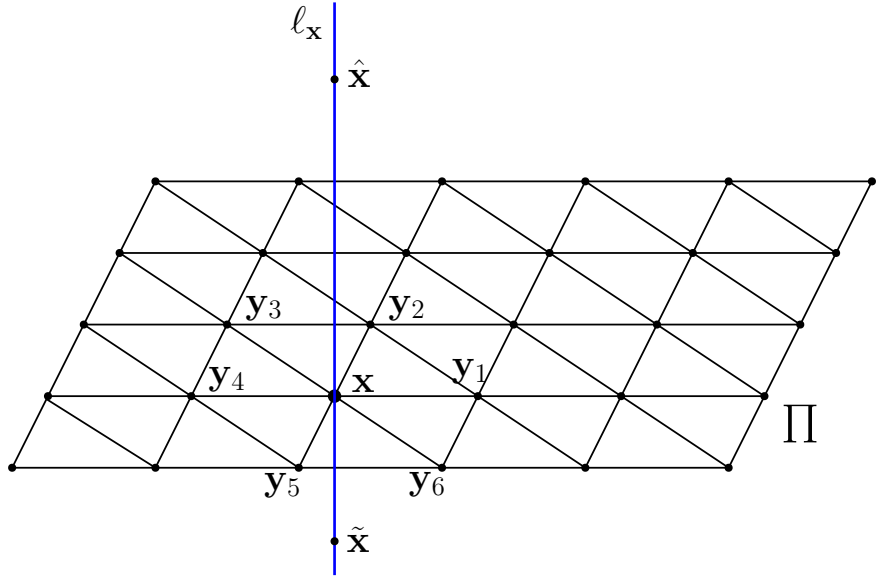}
\caption{The hexagonal lattice $\Lambda$ lying  in the plane $\Pi$. The 6-fold axis $\ell_{\mathbf x}$ is perpendicular  to $\Pi$, and  $\hat { \mathbf x}$ and $\tilde{\mathbf x}$ are points of $X$ on $\ell_{\mathbf x}$ on different sides of $\Pi$ which are  closest to $\mathbf x$ among all points of $X$ in the open half-spaces defined by $\Pi$.}
\label{fig:n=6}
\end{center}
\end{figure}

Now consider the vertex ${\mathbf y}_1$ of the regular hexagon $P$ as well as its two adjacent vertices ${\mathbf y}_2$  and ${\mathbf y}_6$. The $r_1$-shell $H_{{\mathbf y}_1}(r_1)$ at ${\mathbf y}_1$ contains the three points $\mathbf x$, ${\mathbf y}_2$, ${\mathbf y}_6$ from $C_{\mathbf x}(r_1)$.  By arguments similar to those above, each of the three line segments  $[{\mathbf y}_1,{\mathbf x}]$, $[{\mathbf y}_1,{\mathbf y}_2]$, and $[{\mathbf y}_1,{\mathbf y}_6]$ of length $r_1$ must either lie on the axis $l_{{\mathbf y}_1}$ or be perpendicular to this axis. It is clear that this can hold only if the axis $l_{{\mathbf y}_1}$ is perpendicular to each segment and thus to $\Pi$. 

Thus all six vertices of the regular hexagon $P$ centered at the original point $\mathbf x$ have the same property as $\mathbf x$ itself, namely that the 6-fold axes of their $2R$-cluster groups are perpendicular to $\Pi$. This can be rephrased as follows:\ if $P$ is subdivided into six regular triangles by joining each vertex to the center, the resulting tiling of $P$ by triangles has the property that the 6-fold axes at the vertices of each triangle are perpendicular to $\Pi$. 
Now replacing the original point $\mathbf x$ by any of the six vertices of $P$ and applying the same arguments as above then shows that there is a regular hexagon congruent to $P$ centered at each of the six vertices. When subdivided into six regular triangles, each of these regular hexagons overlaps with $P$ on two regular triangles. The collection of regular triangles derived from these hexagons as well as from $P$, forms a patch of a tiling of $\Pi$ by regular triangles. Repeating this process indefinitely using the vertices on the boundary of the patch already constructed, will result in a tiling of $\Pi$ by regular triangles in which the 6-fold axes at the vertices are perpendicular to $\Pi$ . 

Thus $X$ contains the vertex set $\Lambda:=\Lambda_{\mathbf x}$ of a tiling $\mathcal{T}$ of $\Pi$ by regular triangles of side length $r_{1}$, and the 6-fold axes for the $2R$-cluster groups at the points of $\Lambda$ are all perpendicular to $\Pi$. Up to translation (by $\mathbf x$), $\Lambda$ coincides with the hexagonal lattice $\Lambda'$ (say) spanned by the vectors ${\mathbf x}{\mathbf y}_1$ and ${\mathbf x}{\mathbf y}_2$. Note in particular that $\Lambda=X\cap\Pi$. In fact, for any point $\mathbf y\in\Pi$ outside of $\Lambda$ there is a point ${\mathbf x}'\in\Lambda $ at distance less than $r_1$, and clearly $\mathbf y$ does not lie on the $6$-fold axis~$l_{{\mathbf x}'}$; however, this is a contradiction to the definition of $r_{1}$ applied to the $2R$-cluster $C_{{\mathbf x}'}(2R)$. 

Our next goal is to prove that $X$ is (a translate of) either a lattice $\Gamma$ or a bi-lattice 
\[\Gamma\cup(\Gamma+\mathbf t),\;\, {\mathbf t}\not\in\Gamma,\,\mathbf t\perp \Pi,\] 
where in either case the underlying lattice $\Gamma$ is generated by the hexagonal lattice $\Lambda'$ and a vector orthogonal to $\Pi$. Our strategy is to build $X$ layer by layer from its plane section $X\cap\Pi=\Lambda$.
       
We begin by analyzing the shortest distances of points of $X$ in $\Pi$ from points of $X$ in the two open half-spaces of $\mathbb{R}^3$ determined by $\Pi$ and denoted $\tilde\Pi$ and $\hat\Pi$. Let $\tilde{\mathbf x}$ and $\hat{\mathbf x}$ denote points of $X$ closest to ${\mathbf x}$ in the open half-spaces $\tilde\Pi$ and $\hat\Pi$, respectively. Set $\tilde h:= |{\mathbf x}{\tilde {\mathbf x}}|$ and $\hat h:=|{\mathbf x}\hat{\mathbf x}|$, and note that $0<{\tilde h},{\hat h}\leq 2R$ and $\tilde{\mathbf x},\hat{\mathbf x}\in C_{\mathbf x}(2R)$, since otherwise we could find an empty ball of $X$ of radius larger than $R$. Note that, since any two $2R$-clusters of $X$ are equivalent, ${\tilde h}$ and ${\hat h}$ could just as well have been defined relative to any point of $X$ in $\Pi$ in place of $\mathbf x$; in other words, ${\tilde h}$ and ${\hat h}$ are just the shortest distances of points of $X$ in $\Pi$ from points of $X$ in $\tilde\Pi$ and $\hat\Pi$, respectively. In fact, more generally, if $\mathbf y$ is any point of $X$ and $\Pi_{\mathbf y}$ is the plane through $\mathbf y$ perpendicular to $l_{\mathbf y}$, then the shortest distances of $\mathbf y$ from points of $X$ in the two open half-spaces determined by $\Pi_{\mathbf y}$ are precisely ${\tilde h}$ and ${\hat h}$, respectively. 

Without loss of generality we will assume that ${\tilde h}\leq{\hat h}$. We consider the two cases $\tilde h<\hat h$ and $\tilde h= \hat h$ separately, and begin with the more complicated case when $\tilde h < \hat h$.
\vskip.03in
  
\noindent 
{\bf Case 1.}\ $\tilde h < \hat h$ 
   
We will show in this case that $X$ is (a translate of) a bi-lattice. Our first step is to prove that the point $\tilde{\mathbf x}$ of $\tilde\Pi$ associated with $\mathbf x$ lies on the axis $l_{\mathbf x}$ through $\mathbf x$.

First observe that the orthogonal projection of $\tilde{\mathbf x}$ onto $\Pi$ must be a point located in the 2-dimensional Voronoi domain of $\mathbf x$, taken in $\Pi$ with respect to $\Lambda$. Otherwise $\tilde{\mathbf x}$ would be at a distance less than $\tilde h$ from a point ${\mathbf y}\in\Lambda$ distinct from $\mathbf x$, which is impossible by the equivalence of $C_{\mathbf y}(2R)$ and $C_{\mathbf x}(2R)$. Note that the 2-dimensional Voronoi domains in $\Pi$ of points of ${\Lambda}$ are regular hexagons with circumradius $\frac{r_1}{\sqrt 3}<r_1$.     
      
Now suppose that ${\tilde{\mathbf x}}$ does not lie on $l_{\mathbf x}$. Then, by the 6-fold rotational symmetry of $C_{\mathbf x}(2R)$ about $l_{\mathbf x}$, there is a regular hexagon with one vertex at ${\tilde{\mathbf x}}$ and with center on $l_{\mathbf x}$. The circumradius $b$ of this hexagon, and thus its side-length, cannot exceed $\frac{r_1}{\sqrt 3}$, as the projection of $\tilde{\mathbf x}$ onto $\Pi$ lies in the 2-dimensional Voronoi domain of $\mathbf x$. The two vertices of the hexagon adjacent to vertex ${\tilde{\mathbf x}}$ cannot both lie on $l_{\tilde {\mathbf x}}$, and are at distance $b$ from $\tilde {\mathbf x}$. Hence at least one of these vertices, ${\mathbf y}$ (say), 
must satisfy $|{\tilde{\mathbf x}}{\mathbf y}|=b\leq \frac{r_1}{\sqrt 3} <r_1$, which contradicts the definition of $r_{1}$ as the smallest distance between the center and any off-axis point in a $2R$-cluster. 
Bear in mind that the $2R$-clusters of $X$ are mutually equivalent. Thus $\tilde{\mathbf x}$ lies on $\ell_{\mathbf x}$.

We now proceed by investigating the relative positions of the axes $l_{\tilde{\mathbf x}}$ and $l_{\mathbf x}$ both of which pass through $\tilde{\mathbf x}$. Let $\Pi_{\tilde{\mathbf x}}$ denote the plane through $\tilde{\mathbf x}$ perpendicular to~$l_{\tilde {\mathbf x}}$. We now consider two scenarios.  \vskip.03in

\noindent 
{\bf Subcase 1a.}\ The axes $l_{\tilde {\mathbf x}}$ and $l_{\mathbf x}$ are not perpendicular.

We prove that $X$ must be (a translate of) a bi-lattice in this case. 

Our first step is to show that the two axes must coincide, that is, $l_{\tilde{\mathbf x}}=l_{\mathbf x}$. In fact, since $l_{\tilde {\mathbf x}}$ and $l_{\mathbf x}$ are not perpendicular, ${\mathbf x}$ is a point off $\Pi_{\tilde{\mathbf x}}$ and therefore has smallest distance from ${\tilde{\mathbf x}}$, namely $\tilde h$, among all points of $X\setminus\Pi_{\tilde{\mathbf x}}$. By arguments similar to those above, any such point closest to ${\tilde{\mathbf x}}$ must lie on the axis $l_{\tilde{\mathbf x}}$. Hence the axes $l_{\tilde {\mathbf x}}$ and $l_{\mathbf x}$ must actually coincide, and the planes $\Pi_{\tilde{\mathbf x}}$ and $\Pi$ must be parallel. We will refer to the segment $[{\mathbf x},{\tilde{\mathbf x}}]$ as the $\tilde h$-{\it segment\/} of ${\mathbf x}$, as well as of ${\tilde{\mathbf x}}$, and note that the $\tilde h$-segment lies perpendicular to both $\Pi$ and $\Pi_{\tilde{\mathbf x}}$. By the mutual equivalence of the $2R$-clusters of $X$, there exists an $\tilde h$-segment at each point of $X$, and as $\tilde h < \hat h$, this is unique.

Appealing again to the mutual equivalence of $2R$-clusters in $X$, we observe that the plane $\Pi_{\tilde{\mathbf x}}$ through $\tilde{\mathbf x}$ must contain a congruent copy $\Lambda_{\tilde{\mathbf x}}$ of the vertex-set $\Lambda$ of the tessellation $\mathcal{T}$ in $\Pi$. We show that $\Lambda_{\tilde {\mathbf x}}$ is a translate of $\Lambda$, and in particular that  
\[\Lambda_{\tilde {\mathbf x}}=\Lambda+{\mathbf t},\;\,{\mathbf t}:={{\mathbf x}{\tilde{\mathbf x}}}.\]

The $\tilde h$-segment $[{\mathbf y},{\tilde{\mathbf y}}]$ at any point $\mathbf y$ in $\Lambda$, with $\mathbf y$ removed, must lie in one of the open half-spaces $\tilde\Pi$ or $\hat\Pi$. Assign a plus $(+)$ or minus $(-)$ sign to $\mathbf y$ according as its $\tilde h$-segment, with $\mathbf y$ removed, lies in $\tilde\Pi$ or $\hat\Pi$. Clearly, for any triangle of $\mathcal{T}$,  two of its vertices must have the same sign, and we may assume without loss of generality that $\mathbf x$ and ${\mathbf y}_{1}$ are two such vertices and are assigned a plus sign. Thus the endpoints $\tilde{\mathbf x}$ and $\tilde{\mathbf y}_{1}$ of the $\tilde h$-segments at $\mathbf x$ and ${\mathbf y}_{1}$ must lie in $\Lambda_{\tilde{\mathbf x}}$, and 
\[\tilde{\mathbf x}={\mathbf x}+{\mathbf t},\;\,
\tilde{\mathbf y}_{1}={\mathbf y}_{1}+{\mathbf t}.\]
As $l_{\tilde{\mathbf x}}$ is a 6-fold rotation axis for $C_{\tilde{\mathbf x}}(2R)$, the congruent copy $\Lambda_{\tilde{\mathbf x}}$ of $\Lambda$ in $\Pi_{\tilde{\mathbf x}}$ must contain a regular hexagon with one vertex at $\tilde{\mathbf y}_{1}$ and center at $\tilde{\mathbf x}$, which is a translate by $\mathbf t$ of the corresponding hexagon in $\Pi$ with one vertex at ${\mathbf y}_{1}$ and center at ${\mathbf x}$. It follows that $\Lambda_{\tilde{\mathbf x}}$ is actually a translate of $\Lambda$ by $\mathbf t$, and that $\Lambda_{\tilde{\mathbf x}}$ consists of the endpoints of the $\tilde h$-segments of the points in $\Lambda$. 
   \vskip 1cm

We next consider the longer segments $[{\mathbf y},\hat{\mathbf y}]$ of length $\hat h$, called $\hat h$-{\it segments\/}, associated with points $\mathbf y$ of $X$. Then, since the $\tilde h$-segments of the points in $\Lambda$ all lie in $\tilde\Pi$, the $\hat h$-segments of the points of $\Lambda$ must necessarily all lie in $\hat\Pi$. We claim that the $\hat h$-segments also lie perpendicular to $\Pi$. By the mutual equivalence of $2R$-clusters, it is  sufficient to verify this for the $\hat h$-segment at ${\mathbf x}$. 

We prove that $\hat{\mathbf x}$ lies on $l_{\mathbf x}$. The argument is similar as for $\tilde{\mathbf x}$. The projection of $\hat{\mathbf x}$ onto $\Pi$ falls into the Voronoi domain of $\mathbf x$ relative to $\Lambda$, since otherwise the distance of $\hat{\mathbf x}$ from a point $\mathbf y$ of $\Lambda$ distinct from $\mathbf x$ would be smaller than $\hat h$; this is impossible, by the definition of $\hat h$, taken relative to the $2R$-cluster of $\mathbf y$. Now,  if $\hat{\mathbf x}$ does not lie on $l_{\mathbf x}$, then a regular hexagon with one vertex at $\hat{\mathbf x}$ and center on $l_{\mathbf x}$ must appear and its side-length cannot exceed $\frac{r_1}{\sqrt 3}$. This is impossible, for the same reason as in the proof for $\tilde{\mathbf x}$.

Thus the $\hat h$-segments $[{\mathbf y},\hat{\mathbf y}]$ at points $\mathbf y$ of $\Lambda$ also lie perpendicular to $\Pi$, and their endpoints $\hat{\mathbf y}$ form a translate $\Lambda_{\hat{\mathbf x}}$ of $\Lambda$ given by $\Lambda_{\hat{\mathbf x}}=\Lambda+ {\mathbf x}{\hat{\mathbf x}}$. At this point we have constructed three layers of $X$ in parallel planes: $\Lambda$ in $\Pi$; $\Lambda_{\tilde{\mathbf x}}={\mathbf x}{\tilde{\mathbf x}}$ in ${\tilde\Pi}$, at distance $\tilde h$ from $\Pi$; and $\Lambda_{\hat{\mathbf x}}={\mathbf x}{\hat{\mathbf x}}$ in ${\hat\Pi}$, on the other side of $\Pi$, at distance $\hat h$ from $\Pi$.

Finally, the entire Delone set $X$ can be constructed layer by layer proceeding in either direction, beginning with $\Lambda_{\tilde{\mathbf x}}$ to obtain $\Lambda_{\mathbf u}$ with ${\mathbf u}:={\,\widehat{\tilde{\mathbf x}}}$, and similarly with $\Lambda_{\widehat{\mathbf x}}$ to obtain $\Lambda_{\mathbf w}$ with ${\mathbf w}:=\widetilde{\hat{\mathbf x}}$, and so on. 
Thus $X$ consists of parallel layers of (translates of) hexagonal lattices spaced at alternating distances $\tilde h$ and $\hat h$. In particular, up to translation, $X$ is a bi-lattice $\Gamma\cup(\Gamma +{\mathbf t})$, with $\mathbf t$ as above and with underlying 3-dimensional lattice $\Gamma$ spanned by the hexagonal lattice $\Lambda'$ and the vector ${\mathbf x}{\mathbf u}$ (of length ${\tilde h}+{\hat h}$). Thus $X$ is a regular system, and the proof for  Subcase 1a is complete. 
\vskip.03in

\noindent 
{\bf Subcase 1b.}\ The axes $l_{\tilde {\mathbf x}}$ and $l_{\mathbf x}$ are perpendicular.

We claim that this case cannot actually occur. In other words, there is no Delone set which meets this condition.

Suppose there exists a Delone set $X$ such that $l_{\tilde{\mathbf x}}$ and $l_{\mathbf x}$ are perpendicular. Then the plane $\Pi_{\tilde{\mathbf x}}$ at ${\tilde{\mathbf x}}$ must contain $l_{\mathbf x}$ and thus $\mathbf x$ itself. By the mutual equivalence of $2R$-clusters, $\Pi_{\tilde{\mathbf x}}$ must also contain a congruent copy $\Lambda_{\tilde{\mathbf x}}$ of $\Lambda$ with $X\cap \Pi_{\tilde{\mathbf x}}=\Lambda_{\tilde{\mathbf x}}$, and $\Lambda_{\tilde{\mathbf x}}$ must contain the two points $\mathbf x$ and ${\tilde{\mathbf x}}$ at distance ${\tilde h}\leq 2R$. By the 6-fold rotational symmetry of the cluster $C_{\tilde{\mathbf x}}(2R)$, there exists a regular hexagon in $\Lambda_{\tilde{\mathbf x}}$ with one vertex at $\mathbf x$ and with center at $\tilde{\mathbf x}$ obtained by rotation of $\mathbf x$ about $\tilde{\mathbf x}$. Let $\mathbf x '$ be one of two vertices of this hexagon adjacent to vertex $\mathbf x$. Then  $\mathbf x'$ is at distance $\frac{\tilde{h}}{2}$ from $\Pi$. There is point $\mathbf y$ of $\Lambda$ at distance at most $\frac{r_1}{\sqrt 3}$ from the orthogonal projection of $\mathbf x'$ on $\Pi$. Hence 
$$|\mathbf{yx}'|\leq \sqrt{\left(\frac{\tilde{h}}{2}\right)^2+\left(\frac{r_1}{\sqrt{3}}\right)^2}<\tilde{h}.$$
The last inequlity holds because $r_1\leq \tilde{h}$  due to the fact that a copy $\Lambda_{\tilde{x}}$ of $\Lambda$ with smallest interpoint distance $r_1$ contains two points $\mathbf x$ and $\tilde{\mathbf x}$ at distance $\tilde{h}$. Thus, we found points $\mathbf x'\notin \Lambda$ and $\mathbf y\in \Lambda$ at distance less than $\tilde{h}$, which is a contradiction.

It follows that the Delone set $X$ cannot exist, as claimed. This completes the proofs for Subcase 1b as well as Case 1.
\vskip.03in

\noindent 
{\bf Case 2.}\ $\tilde h =\hat h$ 
                  
The arguments in this case are similar to what we saw in Case 1, but the proof is quite a bit simpler and shorter. In particular we show that $X$ is (a translate of) a lattice. 

Set $h:=\tilde h=\hat h$. By arguments similar to those in Case 1, the points $\tilde{\mathbf x}$ and $\hat{\mathbf x}$ closest to $\mathbf x$ among all points of $X$ in $\tilde\Pi$ and $\hat\Pi$, respectively, must lie on $l_{\mathbf x}$ at distance $h$ from $\mathbf x$. Similarly, by the mutual equivalence of $2R$-clusters of $X$, the points $\tilde{\mathbf y}$ and $\hat{\mathbf y}$ associated with any point $\mathbf y$ in $\Lambda$ must lie on $l_{\mathbf y}$ at distance $h$ from $\mathbf y$. But as we saw earlier, the 6-fold axes $l_{\mathbf y}$ at the points $\mathbf y$ of $\Lambda$ all lie perpendicular to $\Pi$. 
Thus the set of endpoints ${\tilde{\mathbf y}}$ of the $\tilde h$-segments $[\mathbf y,{\tilde{\mathbf y}}]$ with $\mathbf y\in\Lambda$ forms a translate $\Lambda_{\tilde{\mathbf x}}$ of $\Lambda$, namely 
$\Lambda_{\tilde{\mathbf x}}=\Lambda +{{\mathbf x}{\tilde{\mathbf x}}}$. Similarly, the set of points ${\hat{\mathbf y}}$ with $\mathbf y\in\Lambda$ forms a translate $\Lambda_{\hat{\mathbf x}}$ of $\Lambda$, namely
$\Lambda_{\hat{\mathbf x}}=\Lambda +{{\mathbf x}{\hat{\mathbf x}}}$.
But $\tilde h=\hat h$, so ${{\mathbf x}{\hat{\mathbf x}}}=-{{\mathbf x}{\hat{\mathbf x}}}$. 

The entire Delone set $X$ can again be constructed layer by layer proceeding in either direction, much in the same way as in Subcase 1a. In this case $X$ is a lattice of the desired kind, up to translation, and thus $X$ is a regular system. This completes the proof for Case 2, and concludes the proof of the theorem.
\end{proof} 
\medskip

The next three theorems, all proved by Dolbilin, coincide with Theorems 3.1, 4.1 and 5.1 of~\cite{Dol2018}, respectively. We follow Sch\"onflies notation for the groups involved.

\begin{theorem}\label{thm:icosahedron}
The $2R$-cluster group of a Delone set in $\mathbb{R}^3$ with $N(2R)=1$ cannot  contain the rotation subgroup $I$ of a regular icosahedron.
\end{theorem}

\begin{theorem}\label{thm:cube}
Let $X$ be a Delone set in $\mathbb{R}^3$ with $N(2R)=1$.
If the cluster group $S_X(2R)$ contains the rotation subgroup $O$ of a cube, then $X$ is a regular system and $S_X(2R)$ also contains the full symmetry group $O_h$ of a cube.
\end{theorem}

\begin{theorem}\label{thm:tetrahedron}
Let $X$ be a Delone set in $\mathbb{R}^3$ with $N(2R)=1$. If $S_X(2R)$ contains the rotation subgroup $T$ of a regular tetrahedron, then $X$ is a regular system and $S_X(2R)$ also contains the full symmetry group $T_d$ of a  regular tetrahedron.
\end{theorem}

The tower bound criterion of Theorem~\ref{thm:tower} was used in \cite{DLS1998}  to obtain the general upper bound for the regularity radius $\hat\rho_d$ in terms of $d$ and $R$ mentioned in the introduction. However, the tower bound may not be sharp for specific groups. For example, if $d=3$ and $S_X(2R)$ is the full symmetry group $O_h$ of a  cube (of order 48), then the tower bound criterion only shows that the equivalence of $14R$-clusters implies regularity of $X$, whereas Theorem~\ref{thm:cube} assures that already the equivalence of $2R$-clusters is enough to imply regularity in this case.

\section{The groups $S_8$ and $D_{4d}$}
\label{S8D4d}

In this section, we deal with two specific finite subgroups of $\rm{O}(3)$, namely $S_8$ and $D_{4d}$ (in Sch\"onflies notation). We show that neither $S_8$ nor $D_{4d}$ occurs as the $2R$-cluster group of a Delone set $X$ in $\R^3$ with $N(2R)=1$.

The group $S_8$ is the subgroup of $\rm{O}(3)$ generated by a rotatory reflection $\sigma$,  defined as the composition of the reflection in a plane $\Pi$ and a rotation by $\frac\pi4$ about a line $\ell$ perpendicular to $\Pi$. With a proper choice of orthogonal coordinate system (so that $\Pi$ is the plane $z=0$, and $\ell$ the line $x=y=0$), and with the rotation component of $\sigma$ in counterclockwise orientation if seen from the half-space $z>0$), the group $S_8$ is given as 
$$S_8:=\left\langle \left(
\begin{array}{ccc}
\cos \frac\pi4 & -\sin\frac\pi4&0\\ 
\sin \frac\pi4 & \cos\frac\pi4&0\\
0&0&-1
\end{array}
\right) \right\rangle.$$
As an abstract group, $S_8$ is isomorphic to a cyclic group of order $8$, but as a geometric group it  differs from the group $C_8$ (again, in Sch\"onflies notation) generated by an eight-fold rotation about a line, which also is a finite subgroup of $\rm{O}(3)$ isomorphic to a cyclic group of order $8$.

The group $D_{4d}$ is generated by $S_8$ and the reflection in a plane $\Pi_1$ which is perpendicular to $\Pi$ and contains $\ell$. Then, with a proper choice of coordinates,  
$$D_{4d}:=\left\langle \left(
\begin{array}{ccc}
\cos \frac\pi4 & -\sin\frac\pi4&0\\
\sin \frac\pi4 & \cos\frac\pi4&0\\
0&0&-1
\end{array}
\right),
\left(
\begin{array}{ccc}
1&0&0\\
0&-1&0\\
0&0&1
\end{array}
\right)
 \right\rangle.$$
Geometrically, the group $D_{4d}$ can be seen as the group of symmetries of a ``regular" antiprism with a square base. (Probably the correct term would be a {\it vertex-transitive} antiprism, but we prefer to call it a regular antiprism meaning a right antiprism with congruent regular bases and a vertex-transitive symmetry group.)

\begin{theorem}\label{thm:s8}
Let $X$ be a Delone set in $\mathbb{R}^3$ with $N(2R)=1$. Then $S_X(2R)$ is not equivalent (conjugate in $\rm{Iso(3)}$) to $S_8$ or $D_{4d}$.
\end{theorem}

\begin{proof}
Suppose $S_X(2R)$ is equivalent to $S_8$ or $D_{4d}$. Without loss of generality we may assume that ${\mathbf x}:=o\in X$, and that $S_X(2R)=S_{\mathbf x}(2R)$ coincides with $S_8$ or $D_{4d}$, with $S_8$ and $D_{4d}$ in the above form. Thus $S_{\mathbf x}(2R)$ contains $S_{8}=\langle\sigma\rangle$, with 
$$\sigma=\left(
\begin{array}{ccc}
\cos \frac\pi4 & -\sin\frac\pi4&0\\
\sin \frac\pi4 & \cos\frac\pi4&0\\
0&0&-1
\end{array}
\right).$$
Let $\ell$ be the line $x=y=0$, and $\Pi$ the plane $z=0$. 

Suppose ${\mathbf y}\neq {\mathbf x}$ is a point of $X$ closest to $\mathbf x$. Let $r_1:=|\mathbf{xy}|$, then $r_1=1$ according to our scaling of $X$, and let $X_1:=H_{\mathbf x}(r_1)$ be the $r_1$-shell of $\mathbf x$ in $X$. Then there are three possibilities for the shell $X_1$, and in each case we can arrive at a contradiction to $r_1$ being the smallest interpoint distance of $X$.

{\bf Case 1.} {\it No point of $X_1$ belongs to $\ell$}. Then the shell $X_1$ splits into orbits under the action of the subgroup $\langle\sigma\rangle$. Since no point of $X_1$ belongs to $\ell$, each such orbit has exactly eight points. If $X_1$ contains more than one orbit, then $X_1$ has at least 16 points, and there are two points of $X_1$ at distance less than $r_1$, as the kissing number in $\R^3$ is 12 (see \cite[Thm.~1.1.1]{Bez2010} for example). This contradicts our assumption on $X$ that any two $2R$-clusters are equivalent.

Thus, $X_1$ is the orbit of a single point $\mathbf y$ under the action of $\langle \sigma \rangle$, so $X_1=\langle\sigma\rangle\cdot\mathbf y$. If a point of $X_1$ lies on $\Pi$, then $\Pi$ contains all eight points of $X_1$ and again some of these points are closer together than~$r_1$. Thus, the eight points of $X_1$ form the vertices of a regular antiprism $P_{\mathbf x}$ centered at $\mathbf x$ (see Fig. \ref{fig:antiprism}).

\begin{figure}[!ht]
\begin{center}
\includegraphics[width=0.5\textwidth]{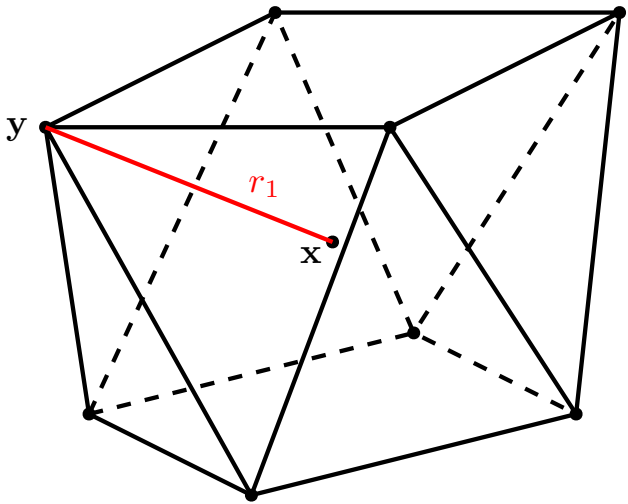}
\caption{The point $\mathbf x$ and the antiprism $P_{\mathbf x}$ generated by the point $\mathbf y$ at the distance $r_1$ from $\mathbf x$.}
\label{fig:antiprism}
\end{center}
\end{figure}

Since the $2R$-cluster of $\mathbf y$ is equivalent to that of $\mathbf x$, we can find a similar antiprism $P_{\mathbf y}$ centered at~$\mathbf y$, and since $|\mathbf{xy}|=r_1$, the point $\mathbf x$ is among the vertices of $P_{\mathbf y}$. In Lemma~\ref{lem:antiprisms} of  the Appendix we show that then there are two vertices of the antiprisms $P_{\mathbf x}$ and $P_{\mathbf y}$ that are at distance less than~$r_1$, which gives a contradiction for this case.

{\bf Case 2.} {\it There are points of $X_1$ on $\ell$, and there are points of $X_1$ not on $\ell$}. As in the previous case, the points of $X_1$ that are not on $\ell$ form the vertices of a single antiprism denoted $P_{\mathbf x}$; we call these points the red points of $X_1$. The shell $X_1$ also contains exactly two points on $\ell$ at distance $r_1$ from $X$; we call these points the blue points of $X_1$. 

We now construct an edge-colored directed graph with vertex set $X$ and with set of directed edges given by the ordered pairs of points of $X$ at distance $r_1$. From every point of $X$, draw red arrows to the eight points of its $r_1$-shell that correspond to red points of $X_1$, and draw blue arrows to the two points of its $r_1$-shell that correspond to blue points of $X_1$. Now, in a large ball, the number of red arrows is approximately four times the number of blue arrows, so there is a pair of points ${\mathbf x}'$ and ${\mathbf y}'$ of $X$ at distance $r_1$ such that both arrows between these two points are red. Using Lemma~\ref{lem:antiprisms} for the antiprisms $P_{{\mathbf x}'}$ and $P_{{\mathbf y}'}$ associated with ${\mathbf x}'$ and ${\mathbf y}'$ in the same manner as $P_{\mathbf x}$ is with $\mathbf x$, we then arrive at a contradiction as before.

{\bf Case 3.} {\it All points of $X_1$ lie on $\ell$.} Without loss of generality we can assume that ${\mathbf y}=(0,0,r_1)$, so that $X_1$ consists of $\mathbf y$ and $-\mathbf{y}=(0,0,-r_1)$.

Let $\mathbf z$ be a point in $C_{\mathbf x}(2R)$ off $\ell$ which has the smallest distance from $\mathbf x$ among all the points in  $C_{\mathbf x}(2R)$ off $\ell$, and let $r_{2}:=|{\mathbf x}{\mathbf z}|$. Recall that the $2R$-clusters in $\mathbf x$ are full-dimensional, so $C_{\mathbf x}(2R)$ indeed contains points off $\ell$. 

If $\mathbf z$ lies in the plane $\Pi$, then the orbit $\langle\sigma\rangle\cdot\mathbf z$ gives a regular octagon in $\Pi$ with circumradius $r_2$. In this case $\mathbf z$, together with the two neighboring vertices in the octagon, forms a non-collinear subset of $X$. As the edge length of the octagon is an interpoint distance of $X$ smaller than the circumradius, this is a contradiction to the property that all point of $X$ closer to $\mathbf x$ than $r_2$ must lie on~$\ell$. Therefore, $\mathbf z$ cannot lie in $\Pi$, and the orbit $\langle\sigma\rangle\cdot {\mathbf z}$ is the vertex-set of a regular antiprism centered at $\mathbf x$; we denote this antiprism by $P'_{\mathbf x}$.

Without loss of generality we can assume that $\mathbf z$ lies on the same side of $\Pi$ as $\mathbf y$. If $\mathbf z$ does not lie on the perpendicular bisector of the segment $[\mathbf x,\mathbf y]$, then Lemma~\ref{lem:antiprisms2} of the Appendix implies that there exist two vertices ${\mathbf u}_1$ and ${\mathbf u}_2$ of the antiprism $P'_{\mathbf y}$ at $\mathbf y$, both in the same base, such that $|\mathbf{z}{\mathbf u}_1|<r_2$ and $|{\mathbf z}{\mathbf u}_2|<r_2$. By the case assumption, the three points $\mathbf z$, ${\mathbf u}_1$, ${\mathbf u}_2$ must be  collinear and lie on the line in the $2R$-cluster of $\mathbf z$ corresponding to $\ell$. However, this is impossible, as the line $\overline{{\mathbf u}_1{\mathbf u}_2}$ is parallel to $\Pi$ and $\mathbf z$ must lie on the same side of the perpendicular bisector of $[\mathbf x,\mathbf y]$ as~$\Pi$. Thus, the only possibility is that $\mathbf z$ is equidistant from $\mathbf x$ and $\mathbf y$ and thus lies on the bisector.

In this case, $P'_{\mathbf x}$ and $P'_{\mathbf y}$ must share a common base. Otherwise, we can use the same arguments as in Lemma \ref{lem:antiprisms2} to find two vertices ${\mathbf u}_1$ and ${\mathbf u}_2$ of one base of $P'_{\mathbf y}$ such that $|\mathbf{z}{\mathbf u}_1|<r_2$ and $|{\mathbf z}{\mathbf u}_2|<r_2$, and then arrive at a contradiction as before.

Thus $\mathbf z$ is a common vertex of $P'_{\mathbf x}$ and $P'_{\mathbf y}$. Note that the $r_2$-shell of $\mathbf z$ contains two points of $X$ at distance $r_1$ from each other, namely $\mathbf x$ and $\mathbf y$. Therefore, the $r_2$-shell of $\mathbf x$ must contain a pair of such points as well. Without loss of generality, we can assume that $\mathbf z$ is one such point in a pair of points ${\mathbf z},{\mathbf z}'$ at distance $r_1$ in the $r_2$-shell of $\mathbf x$. There are two possibilities. If ${\mathbf z}'$ lies on~$\ell$, then ${\mathbf z}'$ must lie between $\mathbf x$ and $\mathbf y$ and hence have smaller distance from $\mathbf x$ or $\mathbf y$ than $r_1$; this is impossible. On the other hand, if ${\mathbf z}'$ is a vertex of $P'_{\mathbf x}$, then another vertex of $P'_{\mathbf x}$ is also at distance $r_1$ from~$\mathbf z$, which contradicts the property that all points of $X$ at distance $r_1$ from $\mathbf z$ must lie on a line through~$\mathbf z$. 
\end{proof}

\section{New upper bound for the regularity radius in $\R^3$}
\label{10rbound}

In this section, we use the results of the previous sections in conjunction with the well-known list of finite subgroups of $\rm{O}(3)$ to obtain a new upper bound for the regularity radius in $\R^3$. We begin by listing all finite subgroups of $\rm{O}(3)$ in Sch\"onflies notation and briefly discuss their relevant properties. We give only one representative in coordinate form when applicable, and describe all subgroups geometrically as well.

\subsection*{List of finite subgroups of $\rm{O}(3)$}
\begin{enumerate}
\item Group $C_n$ is generated by an $n$-fold rotation about a line. For a certain choice of coordinate system, 
$$C_n=
\left\langle \left(
\begin{array}{ccc}
\cos \frac{2\pi}{n} & -\sin\frac{2\pi}{n}&0\\
\sin \frac{2\pi}{n} & \cos\frac{2\pi}{n}&0\\
0&0&1
\end{array}
\right) \right\rangle.$$
The order of $C_n$ is $n$. 

\item Group $S_n$ is generated by a mirror rotation of order $n$, i.e. a composition of an $n$-fold rotation about a line and a mirror reflection in the plane orthogonal to this line. For a certain choice of coordinate system, 
$$S_n=
\left\langle \left(
\begin{array}{ccc}
\cos \frac{2\pi}{n} & -\sin\frac{2\pi}{n}&0\\
\sin \frac{2\pi}{n} & \cos\frac{2\pi}{n}&0\\
0&0&-1
\end{array}
\right) \right\rangle.$$
The order of $S_n$ is $n$ for even $n$, and $2n$ for odd $n$.

\item Group $C_{nh}$ is generated by an $n$-fold rotation about a line and a mirror symmetry in the plane perpendicular to this line. For a certain choice of coordinate system,  
$$C_{nh}=
\left\langle \left(
\begin{array}{ccc}
\cos \frac{2\pi}{n} & -\sin\frac{2\pi}{n}&0\\
\sin \frac{2\pi}{n} & \cos\frac{2\pi}{n}&0\\
0&0&1
\end{array}
\right),
\left(
\begin{array}{ccc}
1 & 0&0\\
0 & 1&0\\
0&0& -1
\end{array}
\right) \right\rangle.$$
The order of $C_{nh}$ is $2n$. If $n$ is odd, then $C_{nh}=S_n$.

\item Group $C_{nv}$ is generated by an $n$-fold rotation about a line and a mirror symmetry in a plane incident to this line. For a certain choice of coordinate system,  
$$C_{nv}=
\left\langle \left(
\begin{array}{ccc}
\cos \frac{2\pi}{n} & -\sin\frac{2\pi}{n}&0\\
\sin \frac{2\pi}{n} & \cos\frac{2\pi}{n}&0\\
0&0&1
\end{array}
\right),
\left(
\begin{array}{ccc}
1 & 0&0\\
0 & -1&0\\
0&0& 1
\end{array}
\right) \right\rangle.$$
The order of $C_{nv}$ is $2n$.

\item Group $D_{n}$ is generated by an $n$-fold rotation about a line, and a reflection in a line (half-turn) which is orthogonal to the first line. For a certain choice of coordinate system, 
$$D_{n}=
\left\langle \left(
\begin{array}{ccc}
\cos \frac{2\pi}{n} & -\sin\frac{2\pi}{n}&0\\
\sin \frac{2\pi}{n} & \cos\frac{2\pi}{n}&0\\
0&0&1
\end{array}
\right),
\left(
\begin{array}{ccc}
1 & 0&0\\
0 & -1&0\\
0&0& -1
\end{array}
\right) \right\rangle.$$
The order of $D_{n}$ is $2n$. 

\item Group $D_{nh}$ is generated by the group $D_n$ and a reflection in the horizontal plane. For a certain choice of coordinate system, 
$$D_{nh}=
\left\langle \left(
\begin{array}{ccc}
\cos \frac{2\pi}{n} & -\sin\frac{2\pi}{n}&0\\
\sin \frac{2\pi}{n} & \cos\frac{2\pi}{n}&0\\
0&0&1
\end{array}
\right),
\left(
\begin{array}{ccc}
1 & 0&0\\
0 & -1&0\\
0&0& -1
\end{array}
\right),
\left(
\begin{array}{ccc}
1 & 0&0\\
0 & 1&0\\
0&0& -1
\end{array}
\right) \right\rangle.$$
The order of $D_{nh}$ is $4n$. Geometrically $D_{nh}$ is the group of all symmetries of a regular prism over an $n$-gon (which is not a cube if $n=4$).

\item Group $D_{nd}$ is generated by the group $D_n$, and a reflection in a plane which contains the axis of rotation for $D_n$ and bisects the angle between two lines of reflection of $D_n$. 
%For a certain choice of coordinate system 
%$$D_{nh}=
%\left\langle \left(
%\begin{array}{ccc}
%\cos \frac\pi{n} & -\sin\frac\pi{n}&0\\
%\sin \frac\pi{n} & \cos\frac\pi{n}&0\\
%0&0&1
%\end{array}
%\right),
%\left(
%\begin{array}{ccc}
%1 & 0&0\\
%0 & -1&0\\
%0&0& -1
%\end{array}
%\right),
%\left(
%\begin{array}{ccc}
%1 & 0&0\\
%0 & 1&0\\
%0&0& -1
%\end{array}
%\right) \right\rangle.$$
The order of $D_{nd}$ is $4n$. Geometrically, if $n\geq 4$, $D_{nd}$ is the group of all symmetries of a regular antiprism over an $n$-gon.

\item The group $T$ of rotational symmetries of a regular tetrahedron, and the group $T_d$ of all symmetries of a regular tetrahedron.

\item The group $T_h$ generated by the group $T_d$ and the central symmetry in $o$.

\item The group $O$ of rotational symmetries of a cube, and the group $O_h$ of all symmetries of a cube.

\item The group $I$ of rotational symmetries of a regular icosahedron, and the group $I_h$ of all symmetries of a regular icosahedron.
\end{enumerate}

Next, we use this classification to improve the best known upper bound for the regularity radius in $\R^3$.

\begin{theorem}
If $X$ is a Delone set in $\R^3$ with $N(10R)=1$, then $X$ is a regular set. Therefore, $\hat\rho_3\leq 10R$.
\end{theorem}

\begin{proof}
Let $S:=S_X(2R)$ be the symmetry group of a $2R$-cluster of $X$. Then $S$ is one of the groups from the list above, so it is enough to consider each case separately. All groups except for $D_{4d}$ were already properly treated in~\cite{Dol2018}, but for the sake of completeness we repeat the main arguments for each group in the list. Note that the application of the previous theorems often requires the condition $N(2R)=1$, or even $N(2jR)=1$ for $j=2$, $3$ or $4$, to hold for $X$, but all these conditions are satisfied by our assumption on $X$ that $N(10R)=1$. 

{\bf Case 1.} $S= C_n$ for some $n$. From Theorem~\ref{thm:rotation} (or \ref{thm:rotation2}) we conclude that $n\leq 6$. Moreover, if $n=6$, then $X$ is regular. The orders of the groups $C_2$, $C_3$, $C_4$, and $C_5$ are 2, 3, 4, and 5 respectively, and each is a product of at most two primes, so the tower bound of Theorem~\ref{thm:tower} implies that $X$ is regular as $N(8R)=1$. 

{\bf Case 2.} $S= S_n$ for some $n$. If $n$ is odd, then $S_n$ contains an $n$-fold rotation about a line, so $n\leq 6$ by Theorem~\ref{thm:rotation}. If $n=1$, $3$, or $5$, the group orders are 2, 6, and 10 respectively, so the tower bound of Theorem~\ref{thm:tower} again implies that $X$ is regular as $N(8R)=1$. This settles the case of odd~$n$. If $n$ is even, then $S_n$ contains an $\frac{n}{2}$-fold rotation, so $n\leq 12$. In this case, if $n$ is not a multiple of~4 (groups $S_2$, $S_6$, and $S_{10}$), then $S_n$ contains a central symmetry and thus,  by Theorem~\ref{thm:antipodal}, $X$ is regular (as $N(2R)=1$). The group $S_{12}$ contains a 6-fold rotation, so then $X$ is regular by Theorem~\ref{thm:n=6}. The group $S_8$ can be eliminated by Theorem~\ref{thm:s8}, as it cannot occur as the symmetry group of a $2R$-cluster as $N(2R)=1$. Finally, the order of the group $S_4$ is 4, so the tower bound of Theorem~\ref{thm:tower} shows again that $X$ is regular as $N(8R)=1$.

{\bf Case 3.} $S= C_{nh}$. The case of odd $n$ was covered in Case 2, since $C_{nh}=S_n$ if $n$ is odd. If $n$ is even, then $C_{nh}$ contains a central symmetry, so again $X$ is regular by Theorem~\ref{thm:antipodal}. 

{\bf Case 4.} $S=C_{nv}$. The group $C_{nv}$ contains an $n$-fold rotation about a line, so Theorem~\ref{thm:rotation} gives $n\leq 6$ and Theorem~\ref{thm:n=6} shows that $X$ is regular if $n=6$. The group orders of $C_{1v}$, $C_{2v}$, $C_{3v}$, $C_{4v}$, and $C_{5v}$ are 2, 4, 6, 8, and 10 respectively, and hence are products of at most three primes. Now the tower bound of Theorem~\ref{thm:tower} shows that $X$ is regular as $N(10R)=1$. 

{\bf Case 5.} $S= D_n$. This case is similar to Case 4. The group $D_n$ contains an $n$-fold rotation about a line, so Theorem~\ref{thm:rotation} shows that $n\leq 6$, and then Theorem~\ref{thm:n=6} that if $n=6$ then $X$ is regular. The group orders of $D_1$, $D_2$, $D_3$, $D_4$, and $D_5$ are 2, 4, 6, 8, and 10 respectively, and hence are products of at most three primes. Again the tower bound shows that $X$ is regular as  $N(10R)=1$.

{\bf Case 6.} $S= D_{nh}$. The group $D_{nh}$ contains an $n$-fold rotation about a line, so $n\leq 6$ by Theorem~\ref{thm:rotation}. If $n$ is even, then $D_{nh}$ also contains a central symmetry, so X is regular by Theorem~\ref{thm:antipodal}. If $n$ is odd, then $n$ is 1, 3, or 5, and the corresponding group orders 4, 12, and 20 are products of at most three primes. In all cases the tower bound of Theorem~\ref{thm:tower} implies that $X$ is regular as $N(10R)=1$.

{\bf Case 7.} $S= D_{nd}$. The group $D_{nd}$ also contains an $n$-fold rotation about a line, so $n\leq 6$, and if $n=6$ then $X$ is regular. If $n$ is odd, then $D_{nh}$ contains a central symmetry and therefore $X$ is regular by Theorem~\ref{thm:antipodal}. The order of $D_{2d}$ is 8, so the tower bound of Theorem~\ref{thm:tower} implies that $X$ is regular. The final group $D_{4d}$ does not occur by Theorem~\ref{thm:s8}. 

{\bf Cases 8.} If $S=T$ or $S= T_d$, then $X$ is regular by Theorem~\ref{thm:tetrahedron}.

{\bf Case 9.} If $S=T_h$, then $S$ contains a central symmetry, so $X$ is regular by Theorem~\ref{thm:antipodal}.

{\bf Case 10.} If $S=O$ or $S=O_h$, then $X$ is regular by Theorem~\ref{thm:cube}.

{\bf Case 11.} The cases of $S=I$ or $S= I_h$ are impossible by Theorem~\ref{thm:icosahedron}.

This concludes the proof of the theorem.
\end{proof}

\section{Concluding remarks}

In this section, we once again revisit the list $L_6$ of all finite subgroups of $\rm{O}(3)$ that only contain a rotation of order at most 6, and list in Table~\ref{tab:groups} the best known upper bound for the regularity radius for each group on $L_6$ that might occur as $2R$-cluster group $S:=S_X(2R)$ in a Delone set $X$ with $N(2R)=1$. We also mention for a given subgroup if it has been shown not to occur as the $2R$-cluster group of a Delone set $X$ with $N(2R)=1$.
\smallskip

\begin{center}
\begin{longtable}{| p{0.1\textwidth} | p{0.1\textwidth} | p{0.2\textwidth} | p{0.4\textwidth} |}
\hline
Group & Order & Regularity radius & Reference\\\hline
$C_1$ & 1 & $4R$ & Tower bound (Theorem \ref{thm:tower})\\ \hline
$C_2$ & 2 & $6R$ & Tower bound (Theorem \ref{thm:tower})\\ \hline
$C_3$ & 3 & $6R$ & Tower bound (Theorem \ref{thm:tower})\\ \hline
$C_4$ & 4 & $8R$ & Tower bound (Theorem \ref{thm:tower})\\ \hline
$C_5$ & 5 & $6R$ & Tower bound (Theorem \ref{thm:tower})\\ \hline
$C_6$ & 6 & $2R$ & Theorem \ref{thm:rotation}\\ \hline

$S_1$ & 2 & $6R$ & Tower bound (Theorem \ref{thm:tower})\\ \hline
$S_2$ & 2 & $2R$ & Theorem \ref{thm:antipodal}\\ \hline
$S_3$ & 6 & $8R$ & Tower bound (Theorem \ref{thm:tower})\\ \hline
$S_4$ & 4 & $8R$ & Tower bound (Theorem \ref{thm:tower})\\ \hline
$S_5$ & 10 & $8R$ & Tower bound (Theorem \ref{thm:tower})\\ \hline
$S_6$ & 6 & $2R$ & Theorem \ref{thm:antipodal}\\ \hline
$S_8$ & 8 & Impossible & Theorem \ref{thm:s8}\\ \hline
$S_{10}$ & 10 & $2R$ & Tower bound (Theorem \ref{thm:tower})\\ \hline
$S_{12}$ & 12 & $2R$ & Theorem \ref{thm:rotation}\\ \hline

$C_{1h}$ & 2 & $6R$ & This is the group $S_1$.\\ \hline
$C_{2h}$ & 4 & $2R$ & Theorem \ref{thm:antipodal}\\ \hline
$C_{3h}$ & 6 & $8R$ & This is the group $S_3$.\\ \hline
$C_{4h}$ & 8 & $2R$ & Theorem \ref{thm:antipodal}\\ \hline
$C_{5h}$ & 10 & $8R$ & This is the group $S_5$.\\ \hline
$C_{6h}$ & 12 & $2R$ & Theorem \ref{thm:antipodal}\\ \hline

$C_{1v}$ & 2 & $6R$ & Tower bound (Theorem \ref{thm:tower})\\ \hline
$C_{2v}$ & 4 & $8R$ & Tower bound (Theorem \ref{thm:tower})\\ \hline
$C_{3v}$ & 6 & $8R$ & Tower bound (Theorem \ref{thm:tower})\\ \hline
$C_{4v}$ & 8 & $10R$ & Tower bound (Theorem \ref{thm:tower})\\ \hline
$C_{5v}$ & 10 & $8R$ & Tower bound (Theorem \ref{thm:tower})\\ \hline
$C_{6v}$ & 12 & $2R$ & Theorem \ref{thm:rotation}\\ \hline

$D_{1}$ & 2 & $6R$ & Tower bound (Theorem \ref{thm:tower})\\ \hline
$D_{2}$ & 4 & $8R$ & Tower bound (Theorem \ref{thm:tower})\\ \hline
$D_{3}$ & 6 & $8R$ & Tower bound (Theorem \ref{thm:tower})\\ \hline
$D_{4}$ & 8 & $10R$ & Tower bound (Theorem \ref{thm:tower})\\ \hline
$D_{5}$ & 10 & $8R$ & Tower bound (Theorem \ref{thm:tower})\\ \hline
$D_{6}$ & 12 & $2R$ & Theorem \ref{thm:rotation}\\ \hline

$D_{1h}$ & 4 & $8R$ & Tower bound (Theorem \ref{thm:tower})\\ \hline
$D_{2h}$ & 8 & $2R$ & Theorem \ref{thm:antipodal}\\ \hline
$D_{3h}$ & 12 & $10R$ & Tower bound (Theorem \ref{thm:tower})\\ \hline
$D_{4h}$ & 16 & $2R$ & Theorem \ref{thm:antipodal}\\ \hline
$D_{5h}$ & 20 & $10R$ & Tower bound (Theorem \ref{thm:tower})\\ \hline
$D_{6h}$ & 24 & $2R$ & Theorem \ref{thm:antipodal}\\ \hline

$D_{1d}$ & 4 & $2R$ & Theorem \ref{thm:antipodal}\\ \hline
$D_{2d}$ & 8 & $10R$ & Tower bound (Theorem \ref{thm:tower})\\ \hline
$D_{3d}$ & 12 & $2R$ & Theorem \ref{thm:antipodal}\\ \hline
$D_{4d}$ & 16 & Impossible & Theorem \ref{thm:s8}\\ \hline
$D_{5d}$ & 20 & $2R$ & Theorem \ref{thm:antipodal}\\ \hline
$D_{6d}$ & 24 & $2R$ & Theorem \ref{thm:rotation}\\ \hline

$T$ & 12 & Impossible & Theorem \ref{thm:tetrahedron}\\ \hline
$T_d$ & 24 & $2R$ & Theorem \ref{thm:tetrahedron}\\ \hline
$T_h$ & 48 & $2R$ & Theorem \ref{thm:antipodal}\\ \hline
$O$ & 24 & Impossible & Theorem \ref{thm:cube}\\ \hline
$O_h$ & 48 & $2R$ & Theorem \ref{thm:antipodal}\\ \hline
$I$ & 60 & Impossible & Theorem \ref{thm:icosahedron}\\ \hline
$I_h$ & 120 & Impossible & Theorem \ref{thm:icosahedron}\\ \hline

\caption{Possible symmetry groups of $2R$-clusters ($2R$-cluster groups), with corresponding bounds for the regularity radius.} 
% needs to go inside longtable environment
\label{tab:groups}
\end{longtable}
\end{center}

From Table~\ref{tab:groups} we can see that any improvement of the $10R$-bound to $8R$ by similar methods would involve an analysis of the five groups $C_{4v}$, $D_4$, $D_{3h}$, $D_{5h}$, $D_{2d}$ and their possible occurrence as $2R$-cluster groups of a Delone set $X$ with $N(X)=1$. The main difference to our treatment of $S_8$ and $D_{4d}$ would lie in the fact that some of these five groups can be realized as the $2R$-cluster groups  of a regular system $X$. For example, the group $C_{4v}$ is the $2R$-cluster group $S_X(2R)$ of the regular system  
%$$X:=\left\{
%\left(\frac13,0,0\right),
%\left(\frac23,0,0\right),
%\left(0,\frac13,0\right),
%\left(0,\frac23,0,0\right),
%\left(0,0,\frac13\right),
%\left(0,0,\frac23\right)
%\right\}+\mathbb{Z}^3.$$
$$X:=\left\{(x,y,z)\in\mathbb{Z}^3\mid z\not\equiv 0\bmod 3\right\}.$$
This means that for $C_{4v}$ a result similar to Theorem~\ref{thm:s8} for $S_8$ and $D_{4d}$ is impossible.  Different types of arguments, probably similar to those of the proof of Theorem~\ref{thm:cube}, are needed to deal with this case.

Another important issue is the existence of $5$-fold rotations as local or global symmetries. It is known that a regular set cannot possess a global $5$-fold rotation, but a similar {\it local\/} restriction has not yet been established. Shtogrin's Theorem~\ref{thm:rotation} says that rotations of order greater than $6$ are forbidden in a Delone set with $N(2R)=1$, but local $5$-fold symmetry still is theoretically possible. Nevertheless, we conjecture that in $\R^3$ neither a regular set nor a non-regular set with $N(2R)=1$ can have local $5$-fold symmetry.

\begin{conjecture}
There is no Delone set $X\subset \mathbb R^3$ such that $N(2R)=1$ and $S_X(2R)$ contains a rotation of order $5$.
\end{conjecture}

Our final remark addresses the more general question of which finite groups actually occur as $2R$-cluster groups in a three-dimensional Delone set with $N(2R)=1$. For a finite subgroup~$G$ of $\rm{O}(3)$ there are three mutually exclusive possibilities.

\begin{enumerate}
\item The group $G$ cannot occur as the $2R$-cluster group of a Delone set $X$ with $N(2R)=1$.
\item The group $G$ occurs as the $2R$-cluster group of a Delone set $X$ with $N(2R)=1$, but only for a regular set $X$.
\item The group $G$ occurs as the $2R$-cluster group of $X$ with $N(2R)=1$ for a non-regular set $X$.
\end{enumerate}
The last case can serve as a source for lower bounds for the regularity radius $\hat\rho_3$. In particular, for every such group $G$ we can ask for the supremum of all numbers $c\geq 2$ such that there is a non-regular set $X$ with $N(cR)=1$ and with $G$ as the $2R$-cluster group of $X$. The example of Engel sets from \cite{LowerBound} shows that for the group $S_1$ (the group generated by a single mirror symmetry) the supremum is $6$, which happens to give the best known lower bound $\hat\rho_3\geq 6R$.

\appendix

\section{Computational details}
\label{appendix}

\begin{lemma}\label{lem:antiprisms}
In the notation of {\bf Case 1} of the proof of Theorem~$\ref{thm:s8}$, there exist two different vertices of the antiprisms $P_{\mathbf x}$ and $P_{\mathbf y}$ that are at distance less than $r_1$.
\end{lemma}
\begin{proof}
Recall that ${\mathbf x}=(0,0,0)$ and $r_1=1$. Without loss of generality we can assume that ${\mathbf y}=(a,0,b)$ with $a^2+b^2=1$ and $a,b\neq 0$. Then the antiprism $P_{\mathbf x}$ can be written as the convex hull of its vertices, 
$$P_{\mathbf x} =\conv \left\{
(\pm a, 0, b),
(0, \pm a, b),
\left(\pm \frac{a}{\sqrt 2}, \pm \frac{a}{\sqrt 2}, -b\right)
\right\},$$
obtained using the action of $\langle \sigma \rangle$ on $\mathbf y$.

Next we exploit the fact that $r_1=1$ is the smallest distance between two points of $X$. Particularly, the two vertices $(a,0,b)$ and $(0,a,b)$ in the same base of $P_{\mathbf x}$ must be at least $1$ apart, and thus 
$$2a^2\geq 1,$$ or equivalently, $a^2\geq b^2$. Similarly, the distance between the two vertices $(a,0,b)$ and $\left(\frac{a}{\sqrt 2}, \frac{a}{\sqrt 2}, -b\right)$ in different bases of $P_\mathbf x$ must be at least $1$, giving $$a^2(1-\sqrt 2)+3b^2\geq 0.$$

Next we will find coordinates for the vertices of $P_\mathbf y$. We know that ${\mathbf x}=(0,0,0)$ is a vertex of $P_\mathbf y$; let $\mathbf z$ be the vertex of $P_\mathbf y$ opposite to $\mathbf x$ in the same base as $\mathbf x$. Suppose $\mathbf z=(a+x,y,b+z)$ for some numbers $x,y,z$, so $\mathbf{yz}=(x,y,z)$. The distance between $\mathbf y$ and $\mathbf z$ must be $r_1=1$, so $$x^2+y^2+z^2=1.$$
Also the scalar product of the vectors $\mathbf{yx}$ and $\mathbf{yz}$ must be equal to the scalar product of the vectors $(a,0,b)$ and $(-a,0,b)$. This turns into the equality 
$$ax+bz=a^2-b^2,$$
which completes our list of conditions for the parameters $a,b,x,y,z$ describing all possible positions of $P_{\mathbf x}$ and $P_\mathbf y$.

The remaining vertices of $P_\mathbf y$ can be computed as follows. The midpoint 
$\mathbf t=\left(\frac{a+x}{2},\frac{y}{2},\frac{b+z}{2}\right)$ 
of the segment $[\mathbf x,\mathbf z]$ is the center of the base of $P_\mathbf y$ that contains $\mathbf x$. The two remaining vertices of this base can be obtained as the endpoints of the vectors 
$$\pm \frac{\mathbf{yx}\times\mathbf{yz}}{2|\mathbf{yt}|},$$ 
placed at $\mathbf t$, where here $\times$ indicates the cross product of vectors. The length of the numerator of this fraction is twice the area of the triangle $\mathbf{xyz}$, and the denominator is twice the height of the same triangle from the vertex $\mathbf y$. Thus, the overall vectors will represent the vectors $\pm\mathbf{tx}$ rotated by $90^\circ$ about the line $\overline{\mathbf{ty}}$. We also note that the denominator is actually the distance from the center of the antiprism $P_\mathbf y$ to its base, which is equal to $|b|$. Thus, the two remaining vertices of $P_\mathbf y$ in this base can be written as 
$$\mathbf t \pm \frac{\mathbf{yx}\times\mathbf{yz}}{2b}\;=\; \left(\frac{a+x}{2},\frac{y}{2},\frac{b+z}{2}\right) \pm \frac{(by,az-bx,-ay)}{2b}.$$

Finally, if we set  
$$\mathbf u := \left(\frac{a+x}{2},\frac{y}{2},\frac{b+z}{2}\right) + \frac{(by,az-bx,-ay)}{2b},$$ 
then the four vertices of the other base of $P_\mathbf y$ can be written as 
$$\mathbf y - \mathbf{yt} + \frac{\pm\mathbf{xt}\pm\mathbf{tu}}{\sqrt2}=\left(\frac{3a-x}{2},-\frac{y}{2},\frac{3b-z}{2}\right)\pm\frac{(a+x,y,b+z)}{2\sqrt 2} \pm \frac{(by,az-bx,-ay)}{2b\sqrt 2}.$$

Now that we know all vertices of the two antiprisms, we can set up the following optimization problem.\\
\indent\indent {\tt Maximize:} minimal distance between vertices of $P_\mathbf x$ and $P_\mathbf y$ that are at least $0.01$ apart. \\
\indent\indent {\tt Under conditions:}
\begin{itemize}
\item $a^2+b^2=1$;
\item $a^2\geq b^2$;
\item $a^2(1-\sqrt 2)+3b^2\geq 0$;
\item $x^2+y^2+z^2=1$;
\item $ax+bz=a^2-b^2$.
\end{itemize}
The (numerical) computations \cite{computations} using Wolfram Mathematica \cite{WM} show that the maximal value of the corresponding function is $0.598$, which means that there is a vertex of $P_\mathbf x$ and a vertex of $P_\mathbf y$ that are at distance at least $0.01$ but are closer than $r_1=1$.
\end{proof}

\begin{lemma}\label{lem:antiprisms2}
In the notation of {\bf Case 3} of the proof of Theorem $\ref{thm:s8}$, if the point $\mathbf z$ does not lie in the perpendicular bisector of the segment $[\mathbf x,\mathbf y]$, then there exist two vertices of the antiprism $P'_{\mathbf y}$ in the same base that are closer to $\mathbf z$ than $r_2$.
\end{lemma}

\begin{proof}
Recall that $\mathbf {\mathbf x}=(0,0,0)$ and $r_1=1$. Without loss of generality we can assume $\mathbf y=(0,0,1)$. We can also assume that $\mathbf z=(a,0,b)$ with $r_2^2=a^2+b^2>1$ and $a\geq 0$, and that $b>0$, because the base of $P'_{\mathbf x}$ that contains $\mathbf z$ is closer to $\mathbf y$ than the other base of $P'_{\mathbf x}$. The point $\mathbf z$ is not on the perpendicular bisector of $[\mathbf{x},\mathbf{y}]$, so $b<\frac 12$; otherwise $r_2$ is not the shortest distance from $\mathbf x$ among all points of $C_{\mathbf x}(2R)$ off $\ell$.
 
The distance between the two vertices $\mathbf z=(a,0,b)$ and $(0,a,b)$ of the same base of $P'_\mathbf {\mathbf x}$ should be at least $r_2$, as otherwise the sides of the base meeting at $\mathbf z$ would give a non-collinear triple of points in $C_{\mathbf z}(2R)$ in which two points are at distance $r_1$ from $\mathbf z$, thus contradicting our case assumption. For similar reasons, the two vertices $\mathbf z=(a,0,b)$ and $\left(\frac{a}{\sqrt 2},\frac{a}{\sqrt 2},-b\right)$ of $P'_{\mathbf x}$ in different bases must also be at least $r_2$ apart. These two conditions give the inequalities
$$a^2\geq b^2,\quad a^2(1-\sqrt2)+3b^2\geq 0.$$

Next we look on the $2R$-cluster of $\mathbf y=(0,0,1)$. The point $\mathbf x$ is at distance $r_1=1$ from $\mathbf y$, so the line corresponding to $\ell$ in $C_{\mathbf y}(2R)$ must coincide with the line $\overline{\mathbf{yx}}=\ell$. Therefore, the line $\ell$ is also the axis of rotation of the antiprism $P'_\mathbf y$ and the two bases of this antiprism lie in the planes $z=1-b$ and $z=1+b$. Without loss of generality we can assume that one of the vertices of $P'_\mathbf y$ is a point $\mathbf u_1:=(x,y,1-b)$ with $x,y\geq 0$. Then $\mathbf u_2:=(y,-x,1-b)$ is a vertex of $P'_\mathbf y$ as well. Note that the distance of $\mathbf u_1$ from $\ell$ is equal to the distance of $\mathbf z$ from $\ell$, so $x^2+y^2=a^2$. 

We claim that 
$$|\mathbf{zu}_1|+|\mathbf{zu}_2|<1+\sqrt{a^2+b^2}=r_1+r_2.$$  
In order to establish this inequality we set up the following optimization problem.\\[.02in]
\indent\indent {\tt Maximize:} $|\mathbf{zu}_1|+|\mathbf{zu}_2|-1-\sqrt{a^2+b^2}$ \\
\indent\indent {\tt Under conditions:}
\begin{itemize}
\item $a^2+b^2>1$;
\item $a\geq 0$;
\item $b>0$;
\item $b<\frac 12$;
\item $a^2\geq b^2$;
\item $a^2(1-\sqrt 2)+3b^2\geq 0$;
\item $x^2+y^2=a^2$;
\item $x\geq 0$;
\item $y\geq 0$.
\end{itemize}
The (numerical) computations \cite{computations} using Wolfram Mathematica \cite{WM} show that the maximal value of the corresponding function (the difference of the two sides of the inequality) is $-0.3367$, which means that for all possible situations,
$$|\mathbf{zu}_1|+|\mathbf{zu}_2|<1+\sqrt{a^2+b^2}=r_1+r_2.$$ 
Since the distance between any two points of $X$ is at least $r_1$, both distances $|\mathbf{zu}_1|$ and $|\mathbf{zu}_2|$ are less than $r_2$. Thus $\mathbf u_1$ and $\mathbf u_2$ are vertices of the antiprism $P'_{\mathbf y}$ in the same base that are closer to $\mathbf z$ than $r_2$.
\end{proof}

\end{document}